%% file: main.tex
\documentclass[11pt]{article}
\usepackage[margin=1in]{geometry}
\usepackage{times}
\usepackage{natbib}

\input{math_commands.tex}

\usepackage[hidelinks]{hyperref}
\usepackage{url}
\usepackage{amssymb, amsthm, amsmath}
\usepackage{booktabs,tabularx,microtype,graphicx,float}

\title{Accelerated Algorithms for Stochastic Monotone Inclusions with Fixed Queries
}

\author{Sucheol Lee \\
AI Center \\
Samsung Electronics \\
Suwon, Republic of Korea \\
\texttt{sucheol0.lee@samsung.com} \\
\and
Donghwan Kim \\
Department of Mathematical Sciences \\
Korea Advanced Institute of Science and Technology \\
Daejeon, Republic of Korea \\
\texttt{donghwankim@kaist.ac.kr}
}

\date{}

\newtheorem{theorem}{Theorem}[section]
\newtheorem{lemma}[theorem]{Lemma}
\newtheorem{proposition}[theorem]{Proposition}
\newtheorem{corollary}[theorem]{Corollary}
\newtheorem{assumption}[theorem]{Assumption}
\theoremstyle{definition}

\theoremstyle{remark}
\newtheorem{remark}[theorem]{Remark}

\DeclareMathOperator{\dist}{dist}
\newcommand{\cH}{\mathcal H}

\newcommand{\Pp}{\mathbb P}

\newcommand{\barL}{\overline L}

\newcommand{\dkedit}[1]{{ #1}}
\newcommand{\scedit}[1]{{ #1}}

\begin{document}

\maketitle

\begin{abstract}
We study acceleration of stochastic first-order methods for monotone Lipschitz inclusions with a fixed number of oracle queries per iteration, measured by the expected squared residual.
Existing methods either have suboptimal oracle complexity, require an increasing number of oracle queries per iteration, or both.
%
\scedit{Under same-sample oracle access and mean-square Lipschitz stochastic noise,} we develop variance-reduced anchored forward-backward (VRAF),
an anytime composite method that
\dkedit{makes two oracle queries per iteration.} 
VRAF attains 
\scedit{the first log-free $O(1/\epsilon)$ complexity
under these oracle assumptions and hence the optimal $O(\sigma^2/\epsilon)$ variance dependence.}
This, however, leaves open
whether the optimal $O(1/\sqrt{\epsilon})$ deterministic term
can also be attained.
We establish 
an impossibility result
by extending the lower bound of \citet{foster2019complexity} to stochastic oracles permitting repeated queries to sampled stochastic operators, showing that $O(LD/\sqrt{\epsilon}+\sigma^2/\epsilon)$ oracle complexity is unattainable in general.
We therefore develop recentered regularized stochastic extragradient (RRSEG),
which attains the near-optimal oracle complexity previously achievable only with an increasing number of queries per iteration, while using a fixed number of queries.
\end{abstract}

\section{Introduction}\label{sec:introduction}
Finding equilibrium points and solving minimax problems are fundamental tasks in optimization and modern machine learning.
Representative examples arise in generative adversarial networks \citep{goodfellow2014generative, arjovsky2017wasserstein},
game theory \citep{rosen1965existence},
multi-agent learning \citep{littman1994markov, mazumdar2025tractable},
and recent approaches to LLM alignment \citep{munos2024nash, xu2025robust}.

In large-scale learning problems, the underlying operator is typically accessible only through a stochastic oracle,
motivating extensive efforts to develop efficient stochastic methods.
Yet our understanding remains limited even for basic monotone Lipschitz inclusions.
For the expected squared residual, a standard optimality measure in this setting,
the oracle complexity has a lower bound of $\Omega(\sigma^2/\epsilon)$ for achieving accuracy $\epsilon$  \citep{foster2019complexity},
where $\sigma^2$ bounds the stochastic oracle variance. However,
no stochastic method is known to attain the optimal 
dependence on $\sigma^2$.
E-Halpern and S-Dual-OHM 
have $O(\sigma^2\epsilon^{-3/2})$
stochastic terms,
while RAIN improves this to the near-optimal $\widetilde O(\sigma^2/\epsilon)$ term
\citep{cai2022stochastic,yoon2026direct,chen2024nearoptimal}.
These methods, however, require an increasing number of oracle queries per iteration as iterations proceed or as $\epsilon$ decreases.
This motivates our first question: \textit{Can the optimal $O(\sigma^2/\epsilon)$ 
stochastic term be attained with a fixed number of oracle queries per iteration?}

We answer this question affirmatively with variance-reduced anchored forward-backward (VRAF), an anytime fixed-query method for monotone Lipschitz inclusions.
VRAF makes exactly two oracle queries per iteration, 
both using the same sampled stochastic operator.
It achieves
$O (((L^2+L_\Delta^2)D^2+\sigma^2)/{\epsilon})$
oracle complexity,
\scedit{where $L$ is the Lipschitz constant of the operator,
$L_\Delta$ is the mean-square Lipschitz constant of the stochastic noise,}
and $D$ is the initial distance.
To our knowledge, this is the first \scedit{log-free} $O(1/\epsilon)$ oracle complexity for the expected squared residual of monotone Lipschitz inclusions
\scedit{under same-sample access and Lipschitz stochastic noise,} even among methods allowing an increasing number of oracle queries per iteration.

A worst-case example shows that the $O(L^2D^2/\epsilon)$ deterministic term of VRAF is order-tight. 
%
This leads to our second question: \textit{Can a different method retain the optimal 
dependence on $\sigma^2$
while also attaining the optimal $O(LD/\sqrt\epsilon)$ deterministic term?}
Such simultaneous optimality
is possible in stochastic convex minimization and, for gap functions, in stochastic variational inequalities \citep{lan2012optimal,juditsky2011solving}.
For the expected squared residual, the corresponding target is
$O (LD / \sqrt{\epsilon} + \sigma^2/\epsilon)$.
We establish 
an impossibility result
by extending the worst-case construction of \citet{foster2019complexity} to stochastic oracles that allow repeated queries to previously sampled stochastic operators, showing that this complexity is unattainable 
even for two-dimensional monotone Lipschitz equations satisfying $L_\Delta\le L$.

This impossibility
motivates our second method, recentered regularized stochastic extragradient (RRSEG).
RRSEG uses two independent oracle queries per iteration and achieves
$
\widetilde O (LD/\sqrt{\epsilon} + \sigma^2/\epsilon)
$
oracle complexity, 
with an optimal deterministic term and a near-optimal stochastic term.
\scedit{
More explicitly, RRSEG has an
$O(\sigma^2\log^3 N/N)$ stochastic term, matching the logarithmic
power in the near-optimal bound of RAIN, while using a fixed, rather than
increasing, number of oracle queries per iteration.
Corollary~\ref{lb:cor:fixed-query} rules out every logarithmic power
$p<1$ for fixed-query methods, leaving the range $1\le p<3$ open.
}

\section{Problem Setting}\label{sec:problem-setting}
Let $\cH$ be a real Hilbert space. 
We consider the monotone Lipschitz inclusion
\begin{align}
    0\in F(z_\star)+A(z_\star),
    \label{eq:problem-inclusion}
\end{align}
where $z_\star\in\cH$.
When $A=0$,~\eqref{eq:problem-inclusion} reduces to $F(z_\star)=0$,
which we call a monotone Lipschitz equation.

\begin{assumption}
\label{ass:problem}
For some $L>0$, the operator $F:\cH\to\cH$ is monotone and $L$-Lipschitz continuous, {\it i.e.,} for all $x,y\in\cH$,
\begin{align*}
    \langle F(x)-F(y),x-y\rangle&\ge0,
    &
    \|F(x)-F(y)\|&\le L\|x-y\|,
\end{align*}
and
$A:\cH\rightrightarrows\cH$ is maximally monotone. 
For a given initial point $z_0$,
there exists a solution $z_\star$ of \eqref{eq:problem-inclusion} such that
$\|z_0-z_\star\|\le D$ for some $D>0$.
\end{assumption}

We access $F$ 
only through its unbiased and bounded stochastic oracle $F_\xi$
and assume exact access to the resolvent $J_{\alpha A}:=(I+\alpha A)^{-1}$ 
for every $\alpha>0$.
\begin{assumption}[Unbiased and bounded stochastic oracle]\label{ass:stochastic-oracle}
For some $\sigma\ge0$ and every $z\in\cH$,
\begin{align*}
    F_\xi(z)&=F(z)+\delta_\xi(z),
    &
    \E[\delta_\xi(z)]&=0,
    &
    \E\|\delta_\xi(z)\|^2&\le\sigma^2.
\end{align*}
\end{assumption}

\scedit{
We also use the following additional oracle conditions where stated.

\begin{assumption}[Same-sample oracle access]\label{ass:same-sample}
For each sampled stochastic operator $F_\xi$, the algorithm may evaluate $F_\xi$ at multiple query points.
\end{assumption}
}

\begin{assumption}[Lipschitz stochastic noise]\label{ass:noise-lipschitz}
There exists $L_\Delta\ge0$ such that, for all $x,y\in\cH$,
\begin{align*}
    \E\|\delta_\xi(x)-\delta_\xi(y)\|^2
    \le L_\Delta^2\|x-y\|^2.
\end{align*}
\end{assumption}

\scedit{
Same-sample oracle access and mean-square Lipschitz continuity of the stochastic oracle are commonly used in stochastic variance reduction~\cite{cai2022stochastic,pethick2023solving}.
Together with the Lipschitz continuity of $F$, Assumption~\ref{ass:noise-lipschitz} is equivalent up to constants to the latter condition.%
\footnote{Under Assumption~\ref{ass:stochastic-oracle}, $\E\|F_\xi(x)-F_\xi(y)\|^2=\|F(x)-F(y)\|^2+\E\|\delta_\xi(x)-\delta_\xi(y)\|^2$.}
Assumptions~\ref{ass:problem}, \ref{ass:stochastic-oracle}, and \ref{ass:noise-lipschitz} therefore imply
\begin{align*}
    \E\|F_\xi(x)-F_\xi(y)\|^2
    \le \barL^2\|x-y\|^2,
\end{align*}
where $\barL^2:=L^2+L_\Delta^2$.
}


Fresh samples are independent.
\scedit{The access in Assumption~\ref{ass:same-sample}
is available when $\xi$ identifies, for example, a retained data sample or a replayable random seed, but not for a fresh-sample-only black-box oracle.}
Each evaluation $F_\xi(z)$ counts as one oracle query; resolvent evaluations are not counted as oracle queries.

%
\scedit{We call a method \textit{fixed-query} if there exists a constant $r$, depending only on the method, such that every iteration uses at most $r$ oracle queries, uniformly over all admissible problem instances, iteration indices, the total number of iterations (horizon), and the target accuracy.}

Define the residual 
$\mathcal R(z):=\dist\left(0,F(z)+A(z)\right)$.
We use the expected squared residual as our 
optimality measure
\citep{cai2022stochastic,pethick2023solving}:
\begin{align*}
\E[\mathcal R(z)^2].
\end{align*}
In the deterministic setting, the corresponding criterion is the squared tangent residual used by \citet{cai2024accelerated}.
When $A=0$, $\mathcal R(z)=\|F(z)\|$.
The oracle complexity is the number of stochastic oracle queries required to produce an output $\widehat z$ satisfying
\begin{align*}
    \E[\mathcal R(\widehat z)^2]\le\epsilon.
\end{align*}

\section{Related Work}\label{sec:related-work}

\paragraph{Deterministic methods.}
For monotone Lipschitz equations, extragradient and Popov's method have $O(L^2D^2/\epsilon)$ oracle complexity for squared residual accuracy $\epsilon$ \citep{korpelevich1976extragradient,popov1980modification,gorbunov2022extragradient,gorbunov2022lastiterate}.
Extragradient-type 
methods extend to monotone Lipschitz inclusions through forward-backward-forward splitting~\citep{tseng2000modified}.
Anchoring-only methods attain the same order,
and P-AGD extends the anchored forward-backward update to monotone Lipschitz inclusions \citep{ryu2020ode,surina2026improved,cai2026lastiterate}.
In these anchoring-only methods, the stepsize is diminishing. By contrast, combining anchoring with an extragradient or Popov-style optimistic correction 
can maintain a nonvanishing stepsize and attain
the optimal $O(LD/\sqrt{\epsilon})$ complexity; 
representative examples include
EAG, FEG, and their extensions \citep{yoon2021accelerated,lee2021fast,cai2024accelerated,alcala2023moving,tran-dinh2021halperntype,cai2023doubly}.
This distinction is relevant here 
because VRAF builds on an anchoring-only forward-backward structure.

\paragraph{Stochastic methods.}
Stochastic extensions of accelerated anchored methods, such as S-FEG and moving-anchor EAG-V, preserve the accelerated deterministic term under prescribed variance decay
but yield a suboptimal stochastic term
\citep{lee2021fast,alcala2025stochastic}. 
More generally, no existing stochastic method
for the expected squared residual attains the optimal $O(\sigma^2/\epsilon)$ term in the oracle complexity.
E-Halpern has $O(\epsilon^{-3/2})$ oracle complexity but uses an increasing number of oracle queries per iteration \citep{cai2022stochastic}.
S-Dual-OHM achieves the same order with a number of oracle queries per iteration that increases as $\epsilon$ decreases, but requires cocoercivity in expectation \citep{yoon2026direct}.
RAIN attains the near-optimal $\widetilde O(LD/\sqrt{\epsilon}+\sigma^2/\epsilon)$ oracle complexity for monotone Lipschitz problems, but is not fixed-query \citep{chen2024nearoptimal}.
Among fixed-query methods, BC-(P)SEG+,
stochastic GOMA, and a single-call
stochastic Halpern method have $O(\epsilon^{-2})$ oracle complexity
\citep{pethick2023solving,sohrabi2026accelerated,kim2026improving}.
VRAF 
attains the optimal $O(\sigma^2/\epsilon)$ term
with a fixed number of oracle queries per iteration.
Table~\ref{tab:stochastic-methods} summarizes the main comparisons.

\input{tables/stochastic_methods}


\paragraph{Deterministic and stochastic optimality.} 
In stochastic convex minimization, the optimal deterministic and stochastic terms can be attained together for function-value accuracy. For the expected squared residual considered here, however, the lower-bound construction of \citet{foster2019complexity} yields an additional logarithmic factor. Section~\ref{sec:acceleration-barrier} extends this construction to stochastic oracles that allow repeated queries to previously sampled stochastic operators, showing that $O(LD/\sqrt{\epsilon}+\sigma^2/\epsilon)$ oracle complexity is unattainable in general.

\paragraph{Recursive regularization.}
To approach the accelerated deterministic rate 
while retaining a near-optimal stochastic rate, 
RAIN employs recursive regularization. 
This technique was
introduced for stochastic convex optimization by \citet{allen-zhu2018how} and was later adapted to 
stochastic minimax optimization
by \citet{chen2024nearoptimal}, who proposed RAIN, 
and to differentially private stochastic variational inequalities
by \citet{bassily2024private}.
RAIN 
solves each regularized subproblem with a stochastic subroutine,
which leads to an increasing number of oracle queries per iteration.
RAIN
also considers a single-loop 
variant 
in its numerical experiments
(Algorithm~9 in \citet{chen2024nearoptimal}), but does not provide a convergence-rate guarantee for this variant.
Building on this single-loop structure,
we develop
RRSEG with horizon-dependent parameters and prove an
$O(L^2D^2/N^2+\sigma^2\log^3N/N)$ 
bound on the
expected squared residual 
using two oracle queries per iteration.

\section{VRAF: Optimal 
Variance Dependence
for Composite Inclusions}\label{sec:vraf}

We now develop variance-reduced anchored forward-backward (VRAF), an anytime fixed-query method for monotone Lipschitz inclusions that attains the optimal $O(\sigma^2/\epsilon)$
dependence on the variance parameter.

As discussed in Section~\ref{sec:related-work},
deterministic acceleration of anchored methods relies on an extragradient or optimistic correction that permits a nonvanishing stepsize. In existing stochastic extensions, the resulting stochastic error terms accumulate, and prescribed variance decay is used to control this accumulation while preserving the accelerated deterministic rate
\citep{lee2021fast,alcala2025stochastic}.
%
%
%
To avoid this additional source of error, VRAF instead uses an anchoring-only
forward-backward update with a diminishing stepsize
\citep{ryu2020ode,surina2026improved,cai2026lastiterate}, 
at the cost of an $O(1/N)$ deterministic term.

VRAF combines this update with
a STORM-type variance-reduced
recursive estimator
\citep{cutkosky2019momentumbased}.
\scedit{Under Assumption~\ref{ass:same-sample}, the variance reduction evaluates}
the same sampled stochastic operator at $z_k$ and $z_{k+1}$;
\scedit{Assumption~\ref{ass:noise-lipschitz} then gives
$\E\|\delta_{\xi_{k+1}}(z_{k+1}) - \delta_{\xi_{k+1}}(z_k)\|^2\le L_\Delta^2\|z_{k+1}-z_k\|^2$.}
This gives the following recursion.

For $\alpha_k>0$ and $\beta_k,\gamma_{k+1}\in(0,1]$, initialize $v_0=F_{\xi_0}(z_0)$ and iterate
\begin{align}
\begin{aligned}
 z_{k+1}
 &= J_{\alpha_kA}\bigl((1-\beta_k)z_k + \beta_kz_0 - \alpha_kv_k\bigr),\\
 v_{k+1}
 &= F_{\xi_{k+1}}(z_{k+1}) + (1-\gamma_{k+1})\bigl(v_k-F_{\xi_{k+1}}(z_k)\bigr).
\end{aligned}
\tag{VRAF}
\label{eq:vraf}
\end{align}

\begin{theorem}\label{thm:inclusion}
Under Assumptions~\ref{ass:problem}--\ref{ass:noise-lipschitz}, VRAF with
 $\alpha_0 = \frac7{12\barL}$,
 $\alpha_{k+1} = \frac{2(k+3)}{2k+7}\alpha_k$,
 $\beta_k = \frac3{k+3}$,
 and
 $\gamma_{k+1} = \frac{4k+13}{4(k+3)(k+4)}$
satisfies, for every integer $N\ge1$,
\begin{align*}
 \E[\mathcal R(z_N)^2]
 \le  35 \frac{\barL^2D^2+\sigma^2}{N+2}.
\end{align*}
\end{theorem}

\begin{remark}\label{rem:vraf-stepsize-decay}
The VRAF stepsize satisfies $\alpha_k=\Theta\big(\frac1{\barL\sqrt{k}}\big)$,
while
$\beta_k=\Theta\big(\frac1k\big)$.
Indeed, since
$
 \sqrt{\frac{k+5/2}{k+7/2}}
 \le \frac{\alpha_{k+1}}{\alpha_k}
 =\frac{k+3}{k+7/2}
 \le \sqrt{\frac{k+3}{k+4}},
$
multiplying from $j=0,\ldots,k-1$ gives
$
 \frac{7}{12\barL}\sqrt{\frac{5}{2k+5}}
 \le \alpha_k
 \le \frac{7}{12\barL}\sqrt{\frac{3}{k+3}}.
$
\end{remark}

Each VRAF iteration uses two oracle queries.
Hence, Theorem~\ref{thm:inclusion} gives $O((\barL^2D^2+\sigma^2)/\epsilon)$ oracle complexity to achieve 
$\E[\mathcal R(z)^2]\leq\epsilon$.
The parameters depend only on $k$ and $\barL$, so VRAF is anytime.
We provide a proof sketch below;
the complete proof is given in Appendix~\ref{app:vraf-proof}.


\begin{proof}[Proof sketch]
For $k\ge 1$, let $h_k\in A(z_k)$ be the element selected by the resolvent step of VRAF, so that
\begin{align*}
 z_k+\alpha_{k-1}h_k
 =(1-\beta_{k-1})z_{k-1}+\beta_{k-1}z_0-\alpha_{k-1}v_{k-1}.
\end{align*}
Define the potential
\begin{align*}
 V_k:=&\frac12\E\left\|
 z_k-z_0
 +\frac{\alpha_k}{\beta_k}\bigl(v_k-F(z_k)\bigr)
 +\frac{\alpha_{k-1}}{\beta_{k-1}}\bigl(F(z_k)+h_k\bigr)
 \right\|^2\\
 &+\frac12\E\left\|
 \frac{\alpha_k}{\beta_k}\bigl(v_k-F(z_k)\bigr)
 -(z_0-z_\star)
 \right\|^2.
\end{align*}
The 
proof establishes
that, for $k\ge 1$,
\begin{align*}
 \E[\mathcal R(z_k)^2]
 & \le 4\left(\frac{\beta_{k-1}}{\alpha_{k-1}}\right)^2 V_k, \\
 V_{k+1}
 &\le \left(1-\frac{2}{5(k+3)}\right)V_k +\frac{2}{5(k+3)}\frac9{10} \left(D^2+\frac{\sigma^2}{\barL^2}\right),\\
 V_1
 &\le \frac9{10}\left(D^2+\frac{\sigma^2}{\barL^2}\right).
\end{align*}
The last two inequalities imply by induction that $V_k\le \frac9{10}\big(D^2+\frac{\sigma^2}{\barL^2}\big)$ for all $k\ge 1$. Combining this bound with the first inequality gives
\begin{align*}
\E[\mathcal R(z_k)^2] \le 4\left(\frac{\beta_{k-1}}{\alpha_{k-1}}\right)^2 \frac9{10}\left(D^2+\frac{\sigma^2}{\barL^2}\right).
\end{align*}
Finally, the parameter choices satisfy
$
 \big(\frac{\beta_{k-1}}{\alpha_{k-1}}\big)^2
 \le \frac{175\barL^2}{18(k+2)}.
$
Setting $k=N$ proves the theorem.
\end{proof}

The $O(L^2D^2/N)$ deterministic term in 
Theorem~\ref{thm:inclusion}
is order-tight for the VRAF.
This already holds for the one-dimensional linear operator $F(z)=Lz$ with $A=0$;
see Appendix~\ref{app:vraf-deterministic-tightness}.
Hence, 
the VRAF cannot attain the optimal $O(L^2D^2/N^2)$ deterministic rate.

\section{
Impossibility of achieving optimal deterministic and stochastic rates simultaneously}
\label{sec:acceleration-barrier}
Having established the optimal $O(\sigma^2/\epsilon)$ 
variance dependence with VRAF, we next ask whether 
it can be attained together
with the optimal $O(LD/\sqrt{\epsilon})$ deterministic rate.
We answer this question negatively with the following lower bound theorem.

\begin{theorem}[Impossibility of simultaneous optimality]\label{lb:thm:main}
Under \scedit{Assumptions~\ref{ass:problem}, \ref{ass:stochastic-oracle}, and \ref{ass:noise-lipschitz},}
even when restricted to two-dimensional monotone Lipschitz equations with $A=0$ and $L_\Delta\le L$,
it is impossible to uniformly guarantee
$\E[\mathcal R(\widehat z)^2]\le\epsilon$
using
$O\big(\frac{LD}{\sqrt{\epsilon}} + \frac{\sigma^2}{\epsilon}\big)$
oracle queries.
This impossibility holds even when previously sampled stochastic operators may be queried repeatedly.
\end{theorem}

We provide a proof sketch below; the complete proof is given in Appendix~\ref{app:barrier-proof}.


\begin{proof}[Proof sketch]
Let $F_Z$ be a rescaled version of the worst-case family of \citet{foster2019complexity}.
Appendix~\ref{app:barrier-proof} extends this family so that $L=L_\Delta=3$, $D=1$, and $\sigma^2=\frac1{M^2}$ for $M\ge 16$.
At $\epsilon_M=\frac1{256M^4}$, by the lower bound construction of \citet{foster2019complexity} and this extension, achieving $\E[\mathcal R(\widehat z)^2]\le\epsilon_M$ requires $\Omega(M^2\log M)$ oracle queries, while $LD/\sqrt{\epsilon_M}+\sigma^2/\epsilon_M=304M^2$.
\end{proof}




For fixed-query methods, the $\Omega(M^2\log M)$ lower bound readily translates into a corresponding lower bound on the iteration count $N$
as below.

\begin{corollary}[
Necessity of a logarithmic factor for fixed-query methods]
\label{lb:cor:fixed-query}
For every $0\le p<1$, no fixed-query method can uniformly guarantee, for all integers $N\ge2$,
\begin{align*}
 \E[\mathcal R(\widehat z_N)^2]
 =O\left(
 \frac{L^2D^2}{N^2}
 +\frac{\sigma^2\log^p N}{N}
 \right).
\end{align*}
\end{corollary}
The proof is given in Appendix~\ref{app:barrier-proof}.
\scedit{
Corollary~\ref{lb:cor:fixed-query} leaves the range $1\le p<3$ open:
it rules out every logarithmic power $p<1$, whereas the near-optimal
bounds of both RAIN and RRSEG have stochastic logarithmic power $p=3$;
RRSEG achieves this dependence with a fixed number of oracle queries per iteration.
}

\section{RRSEG: 
Optimal deterministic and near-optimal stochastic rates
}\label{sec:rrseg}
\input{sections/rrseg_fbf_main.tex}

\section{Experiments}\label{sec:experiments}
\input{sections/experiments_main.tex}

\section{Discussion}

\scedit{
The main open question left by Theorem~\ref{lb:thm:main} is the tradeoff between the deterministic and stochastic terms.
In particular, it remains open whether the exact $O(\sigma^2/N)$ stochastic term can be attained while improving the deterministic term  $O(1/N)$,
and whether
\scedit{the logarithmic power in the
$O(\sigma^2\log^3N/N)$ stochastic term 
can be reduced toward the lower-bound threshold while preserving the optimal
$O(1/N^2)$ deterministic term under fixed-query access.}

Our guarantees also rely on monotonicity.
Although deterministic acceleration extends to negative comonotonicity
\citep{lee2021fast,gorbunov2023convergence},
corresponding stochastic guarantees are unknown;
the RPS experiment therefore lies outside our theory.

Finally, residual convergence does not imply pointwise convergence.
Appendix~\ref{app:further-discussion} gives a counterexample for VRAF, and the RRSEG guarantees do not address pointwise convergence.



}

\section{Conclusion}

We studied fixed-query acceleration for stochastic monotone Lipschitz inclusions under the expected squared residual criterion. 
VRAF is the first method to attain the optimal $O(\sigma^2/\epsilon)$ variance dependence,
even among methods allowing an increasing number of oracle queries per iteration, while itself using only a fixed number of queries. 
We also established a lower bound showing that the optimal deterministic and stochastic rates cannot, in general, be attained simultaneously, even when $L_\Delta\le L$ and
repeated queries to previously sampled stochastic operators are allowed.
RRSEG attains the near-optimal $\widetilde O(LD/\sqrt{\epsilon}+\sigma^2/\epsilon)$ oracle complexity previously achieved with increasing 
oracle queries, while using only a fixed number of oracle queries per iteration.

\newpage

\subsection*{AI use statement}
In this work, we used generative AI tools to assist with algorithmic exploration,
mathematical analysis and proof development, numerical experiments,
literature search, code development, translation, and manuscript editing.
The research objectives and major methodological choices were determined by the authors.
All AI-assisted work was reviewed and independently checked by the authors,
who take responsibility for the final content of this work.

\subsection*{Reproducibility statement}
The appendix contains complete proofs of the VRAF, RRSEG,
and pointwise results, together with the lower-bound transfer.
For the experiments, the appendix specifies the exact problems, schedules, tuning procedures, and evaluation protocol.



\bibliography{references}
\bibliographystyle{plainnat}

\appendix

\section{Further related work}\label{app:further-related}
\input{appendices/further_related_work.tex}

\subsection{RAIN under the expected squared residual}\label{app:rain-squared-residual}
\input{appendices/rain.tex}


\section{Complete VRAF Proof}\label{app:vraf-proof}
\input{appendices/vraf_full_proof.tex}

\section{Proofs for Impossibility of Simultaneous Optimality}\label{app:barrier-proof}
\input{appendices/acceleration_barrier_proof.tex}

\section{Complete RRSEG proof}\label{app:rrseg-proof}
\input{appendices/rrseg_fbf_proof.tex}

\section{Experimental details}\label{app:experiments}
\input{appendices/experiments_full.tex}

\section{Further discussion}\label{app:further-discussion}
\input{appendices/discussion_full.tex}

\end{document}

%% file: math_commands.tex
\usepackage{amsmath,amsfonts,bm}

\def\eqref#1{equation~\ref{#1}}

\def\1{\bm{1}}

\DeclareMathAlphabet{\mathsfit}{\encodingdefault}{\sfdefault}{m}{sl}
\SetMathAlphabet{\mathsfit}{bold}{\encodingdefault}{\sfdefault}{bx}{n}

\newcommand{\E}{\mathbb{E}}

\newcommand{\R}{\mathbb{R}}



%% file: tables/stochastic_methods.tex
\begin{table}[b!]
\caption{
Oracle complexity for $\E[\mathcal R(z)^2]\le\epsilon$.
The Composite column indicates whether the guarantee applies to
$A\neq0$; 
$\triangle$ denotes constrained problems only.
In the Output column,
Terminal denotes a nonrandomized terminal output; 
for RRSEG, this is the corrected output $\widehat{z}_N$ in Theorem~\ref{rr:thm:main}.
The notation $\widetilde O$ hides logarithmic factors.
${}^\ast$
Appendix~\ref{app:rain-squared-residual} derives the displayed expected-squared-residual complexity from the proof of \citet{chen2024nearoptimal},
whose theorem is stated for the expected residual.
${}^\dagger$The prescribed variance decay is translated into oracle complexity using independent increasing batches under Assumption~\ref{ass:stochastic-oracle}.
\scedit{VRAF requires Assumption~\ref{ass:same-sample} and~\ref{ass:noise-lipschitz};}
RRSEG 
requires neither.
}
\label{tab:stochastic-methods}
\begin{center}
\begin{tabular}{lcccc}
\toprule
Method & Fixed-query & Output & Oracle complexity & Composite \\
\midrule
Moving-anchor EAG-V & & Terminal & $O\left(\frac{LD}{\sqrt{\epsilon}}+\frac{\sigma^2L^3D^3}{\epsilon^{5/2}}\right){}^\dagger$ & \\
S-FEG & & Terminal & $O\left(\frac{LD}{\sqrt{\epsilon}}+\frac{\sigma^2L^2D^2}{\epsilon^2}\right){}^\dagger$ & \\
BC-(P)SEG+ & $\checkmark$ & Randomized & $\widetilde O\left(\frac{\barL^4D^4+\sigma^4}{\epsilon^2}\right)$ & $\checkmark$ \\
Stochastic GOMA & $\checkmark$ & Terminal & $O\left(\frac{L^4D^4+\sigma^4}{\epsilon^2}\right)$ & \\
E-Halpern & & Terminal & $O\left(\frac{\barL^3D^3+\sigma^2\barL D}{\epsilon^{3/2}}\right)$ & $\triangle$ \\
RAIN & & Randomized & $\widetilde O\left(\frac{LD}{\sqrt{\epsilon}}+\frac{\sigma^2}{\epsilon}\right){}^\ast$ & \\
\midrule
VRAF (ours) & $\checkmark$ & Terminal & $O\left(\frac{\barL^2D^2+\sigma^2}{\epsilon}\right)$ & $\checkmark$ \\
RRSEG (ours) & $\checkmark$ & Terminal & $\widetilde O\left(\frac{LD}{\sqrt{\epsilon}}+\frac{\sigma^2}{\epsilon}\right)$ & \checkmark \\
\bottomrule
\end{tabular}
\begin{minipage}{0.98\linewidth}
\end{minipage}
\end{center}
\end{table}

%% file: sections/rrseg_fbf_main.tex
The impossibility result
shows that $O(LD/\sqrt{\epsilon}+\sigma^2/\epsilon)$ oracle complexity is unattainable in general.
We therefore ask whether a fixed-query method can retain the optimal deterministic rate 
while losing only logarithmic factors in the stochastic rate.

RAIN achieved the near-optimal rate using recursive regularization \citep{chen2024nearoptimal}.
Let $H_k(z):=F(z)+a_k(z-c_k)$ and let $r_k$ be the unique zero of $A+H_k$, with $c_0=z_0$.
Since $\mathcal R(r_0)\le a_0D$, choosing a small $a_0$ gives a small residual for the original inclusion but also a small strong-monotonicity constant.\footnote{An operator $H$ is $\mu$-strongly monotone if $\langle H(x)-H(y),x-y\rangle\ge \mu\|x-y\|^2$ for all $x,y$.}
If $r_k$ were available, adding $\lambda_k(z-r_k)$ to $H_k$ would increase the strong-monotonicity constant while preserving 
$r_k$ as the zero of the regularized inclusion.

However, RAIN implements
this recursive regularization by solving each regularized subproblem with a stochastic subroutine \citep{chen2024nearoptimal}.
\citet{chen2024nearoptimal} also consider a single-loop variant with one stochastic extragradient step per update, but without providing a convergence rate analysis.
Building on this single-loop structure, RRSEG
applies one stochastic forward-backward-forward step to each regularized inclusion and then recenters the added regularization at the resulting iterate. 
When $A=0$, this step reduces to stochastic extragradient. 
With horizon-dependent parameters, RRSEG
retains the optimal deterministic rate while 
incurring
only logarithmic factors in the stochastic rate.
Unlike VRAF, 
its
analysis uses only Assumptions~\ref{ass:problem} and \ref{ass:stochastic-oracle};
\scedit{in particular, it requires neither Assumption~\ref{ass:same-sample} nor Assumption~\ref{ass:noise-lipschitz}.}

Since $r_k$ is unavailable, RRSEG replaces it with the one-step approximation $z_{k+1}$:
\begin{align*}
 H_{k+1}(z)=H_k(z)+\lambda_k(z-z_{k+1}).
\end{align*}
This gives the following single-loop recursion.

For $a_0>0$, $c_0=z_0$, $\eta_k>0$, and $\lambda_k\ge0$, draw independent samples $\xi_k$ and $\zeta_k$ and iterate
\begin{align}
\begin{aligned}
 y_k&=J_{\eta_k A}\bigl(z_k-\eta_k\bigl(F_{\xi_k}(z_k)+a_k(z_k-c_k)\bigr)\bigr),\\
 z_{k+1}&=y_k-\eta_k\bigl(F_{\zeta_k}(y_k)+a_k(y_k-c_k)-F_{\xi_k}(z_k)-a_k(z_k-c_k)\bigr),\\
 a_{k+1}&=a_k+\lambda_k,\\
 c_{k+1}&=\frac{a_kc_k+\lambda_kz_{k+1}}{a_{k+1}}.
\end{aligned}
\tag{RRSEG}
\label{rr:eq:rrseg}
\end{align}
For $a_0\in(0,L)$ and $0<\gamma\le\frac13$, define
$Q_{N,\gamma}(a_0):=\frac{\frac{\gamma N}{3}-(L/a_0-1)}{\log(L/a_0)}$.
The theorem below gives the exact finite-horizon bound. In particular, the concrete parameter choices
in Corollary~\ref{rr:cor:explicit} yields
the simplified bound $O\bigl(\frac{L^2D^2}{N^2}+\frac{\sigma^2\log^3 N}{N}\bigr)$.

\begin{theorem}[RRSEG]\label{rr:thm:main}
Under Assumptions~\ref{ass:problem} and~\ref{ass:stochastic-oracle}, fix an integer $N\ge2$, $a_0\in(0,L)$, and $0<\gamma\le\frac13$ such that $Q_{N,\gamma}(a_0)\ge1$.
Choose the RRSEG parameters as
$\eta_k=\frac1{3(L+Q_{N,\gamma}(a_0)a_k)}$,
$a_{k+1}=\min\{L,\frac{a_k}{1-\gamma\eta_ka_k}\}$, and $\lambda_k:=a_{k+1}-a_k$.
Then $a_N=L$.
Let $m:=\lfloor\min\{Q_{N,\gamma}(a_0),N\}\rfloor$,
$\overline F_N:=\frac1m\sum_{k=N-m}^{N-1}F_{\xi_k}(z_k)$,
and define the output
$\widehat z_N:=J_{\frac1L A}\big(z_N-\frac1L\overline F_N\big)$.
Then RRSEG satisfies
\begin{align}
 \left(\E[\mathcal R(\widehat z_N)^2]\right)^{1/2}
 \le&\sqrt2LD\left[
 2\frac{a_0}{L}
 +\left(\sqrt2-1+\frac{77}{32}\right)
 \left(\frac{a_0}{L}\right)^2
 \right]
 +\frac{\sqrt2\sigma}{\sqrt m}\notag\\
 &+\sigma\sqrt{\frac{5}{2(2-3\gamma)(Q_{N,\gamma}(a_0)+1)}}
 \left(\log\frac{L}{a_0}+\sqrt2+\frac94\right).
 \label{rr:eq:rms-main}
\end{align}
\end{theorem}

The average $\overline F_N$ uses oracle evaluations already generated by RRSEG, so computing $\widehat z_N$ requires no additional stochastic oracle queries.

We provide a proof sketch below; the complete proof is given in Appendix~\ref{app:rrseg-proof}.

\begin{proof}[Proof sketch]
For each $k$, let $r_k$ be the zero of $A+H_k$.
The proof establishes the following bounds on the distances from the iterates to these zeros:
\begin{align*}
 \E\|z_k-r_k\|^2
 &\le\left(\frac{a_0}{a_k}\right)^4D^2
 +\frac{5\sigma^2}{4(2-3\gamma)(Q_{N,\gamma}(a_0)+1)a_k^2},\\
 \E\|z_{k+1}-r_k\|^2
 &\le\left(\frac{a_0}{a_{k+1}}\right)^4D^2
 +\frac{5\sigma^2}{4(2-3\gamma)(Q_{N,\gamma}(a_0)+1)a_{k+1}^2}.
\end{align*}
With $p_N:=J_{\frac1L A}\left(z_N-\frac1L F(z_N)\right)$, the recentering relation and firm nonexpansiveness of the resolvent further give
\begin{align*}
 L\left(\E\|z_N-p_N\|^2\right)^{1/2}
 \le a_0D
 +\sum_{k=0}^{N-1}\lambda_k\left(\E\|z_{k+1}-r_k\|^2\right)^{1/2}
 +\sqrt2L\left(\E\|z_N-r_N\|^2\right)^{1/2}.
\end{align*}
The choice of $Q_{N,\gamma}(a_0)$ gives $a_N=L$.
Then, using $\sqrt{u+v}\le\sqrt u+\sqrt v$, the above inequalities give
\begin{align}
 &L\left(\E\|z_N-p_N\|^2\right)^{1/2}\notag\\
 &\le\left(a_0+a_0^2\sum_{k=0}^{N-1}\frac{\lambda_k}{a_{k+1}^2}+\sqrt2\frac{a_0^2}{L}\right)D
 +\sigma\sqrt{\frac{5}{4(2-3\gamma)(Q_{N,\gamma}(a_0)+1)}}
 \left(\sum_{k=0}^{N-1}\frac{\lambda_k}{a_{k+1}}+\sqrt2\right)\notag\\
 &\le\left(2a_0+(\sqrt2-1)\frac{a_0^2}{L}\right)D
 +\sigma\sqrt{\frac{5}{4(2-3\gamma)(Q_{N,\gamma}(a_0)+1)}}
 \left(\log\frac{L}{a_0}+\sqrt2\right), \label{rr:eq:pN-bound}
\end{align}
where the last inequality follows from
\begin{align*}
 \sum_{k=0}^{N-1}\frac{\lambda_k}{a_{k+1}^2}
 &\le\int_{a_0}^{L}\frac{da}{a^2}=\frac1{a_0}-\frac1L,
 &
 \sum_{k=0}^{N-1}\frac{\lambda_k}{a_{k+1}}
 &\le\int_{a_0}^{L}\frac{da}{a}=\log\frac{L}{a_0}.
\end{align*}
The proof also gives
\begin{align*}
 \left(\E\|\overline F_N-F(z_N)\|^2\right)^{1/2}
 \le\frac{77}{32}\frac{a_0^2}{L}D+\frac{\sigma}{\sqrt m}
 +\frac94\sigma\sqrt{\frac{5}{4(2-3\gamma)(Q_{N,\gamma}(a_0)+1)}}.
\end{align*}
Finally,
\begin{align*}
 \left(\E[\mathcal R(\widehat z_N)^2]\right)^{1/2}
 \le\sqrt2L\left(\E\|z_N-p_N\|^2\right)^{1/2}
 +\sqrt2\left(\E\|\overline F_N-F(z_N)\|^2\right)^{1/2}.
\end{align*}
Substituting the preceding bounds gives \eqref{rr:eq:rms-main}; the complete calculation is given in Appendix~\ref{app:rrseg-proof}.
\end{proof}

\begin{corollary}[RRSEG with concrete parameter choices]\label{rr:cor:explicit}
Under the assumptions of Theorem~\ref{rr:thm:main}, let $N\ge19$ and set
$\gamma=\frac13$, $a_0=\frac{18L}{N}$,
$Q_N=\frac{N/18+1}{\log(N/18)}$, and
$m_N=\lfloor\min\{Q_N,N\}\rfloor$.
Then $Q_N>1$, and RRSEG 
satisfies
\begin{align*}
 \left(\E[\mathcal R(\widehat z_N)^2]\right)^{1/2}
 \le&\frac{36\sqrt2LD}{N}
 +\frac{324\sqrt2LD}{N^2}\left(\sqrt2-1+\frac{77}{32}\right)
 +\frac{\sqrt2\sigma}{\sqrt{m_N}}\\
 &+\sigma\sqrt{\frac{5}{2(Q_N+1)}}
 \left(\log\frac{N}{18}+\sqrt2+\frac94\right).
\end{align*}
Consequently,
$\E[\mathcal R(\widehat z_N)^2]
=O\big(\frac{L^2D^2}{N^2}+\frac{\sigma^2\log^3N}{N}\big)$.
\end{corollary}

\begin{proof}
For $N\ge19$, the displayed choice satisfies $a_0\in(0,L)$ and $Q_N>1$; Appendix~\ref{app:rrseg-proof} gives the elementary verification of the latter inequality.
Substituting $\gamma=\frac13$ and $a_0=\frac{18L}{N}$ into Theorem~\ref{rr:thm:main} gives the upper bound and the stated order.
\end{proof}

Corollary~\ref{rr:cor:explicit} is consistent with Theorem~\ref{lb:thm:main}, since the stochastic term is optimal only up to logarithmic factors.
\scedit{
Repeating RRSEG from the same initial point with doubling horizons 
yields an anytime method without changing the asymptotic rate; see Appendix~\ref{app:rrseg-anytime}.}

When $A=0$, $p_N=z_N-\frac1L F(z_N)$ and hence $L(z_N-p_N) = F(z_N)$.
Therefore, \eqref{rr:eq:pN-bound} directly bounds $(\E\|F(z_N)\|^2)^{1/2}$,
so the last iterate $z_N$ attains the same oracle-complexity order without the final output correction.

Appendix~\ref{app:rrseg-schedule} illustrates the horizon-dependent schedule used in Corollary~\ref{rr:cor:explicit} and in the experiments.



%% file: sections/experiments_main.tex
Figure~\ref{fig:main-experiments} compares VRAF (Section~\ref{sec:vraf}), RRSEG (Section~\ref{sec:rrseg}), BC-(P)SEG+ \citep{pethick2023solving}, stochastic GOMA \citep{sohrabi2026accelerated}, RAIN-SL \citep{chen2024nearoptimal}, and E-Halpern \citep{cai2022stochastic}.
Problem~3 has $A\neq0$ but is not a constrained problem, so only VRAF, RRSEG, and BC-PSEG+ are evaluated.
Each uses at most
$10^7$ oracle queries. 
%
%
VRAF and RRSEG use their theoretical parameters
without empirical tuning.
VRAF (practical) 
replaces $\alpha_k$ by $4\alpha_k$;
it is heuristic, and Appendix~\ref{app:vraf-stepsize-sensitivity} compares multipliers $1,2,4,$ and $8$.
RAIN-SL is the tuned Algorithm~9 of \citet{chen2024nearoptimal}.
E-Halpern is calibrated through its target error $\epsilon$.
Appendix~\ref{app:experiments} gives the remaining details.

\begin{figure}[!ht]
 \centering
 \includegraphics[width=.9\linewidth]{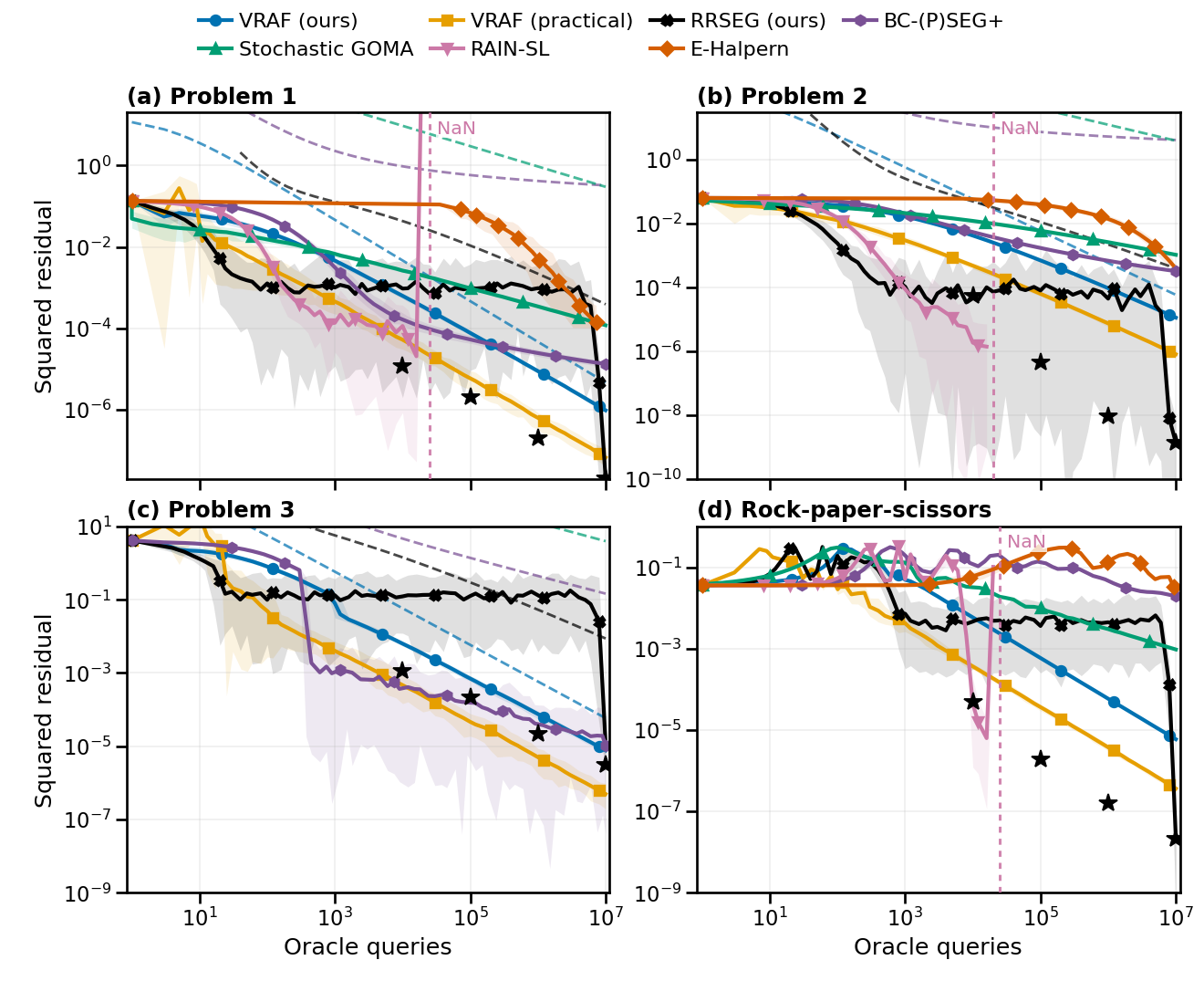}
 \caption{Squared residual versus the number of oracle queries.
 Solid curves show the mean over $50$ independent runs, and shaded regions show the pointwise $2.5\%$--$97.5\%$ quantiles.
 Dashed curves show the theoretical bounds described in the text.
 \scedit{RRSEG solid curves show intermediate outputs from a single horizon-dependent run, whereas stars show terminal outputs from separate horizon-matched runs.}}
 \label{fig:main-experiments}
\end{figure}

A single RRSEG run is horizon-dependent, and E-Halpern uses a prescribed target accuracy.
Their solid curves show the corresponding intermediate iterates.
For RRSEG, stars show terminal outputs from separate horizon-matched runs; on Problem~3, they show the corrected outputs $\widehat{z}_N$.
The RRSEG dashed curve is horizon-matched, using \eqref{rr:eq:equation-bound} on Problems~1 and~2 and Corollary~\ref{rr:cor:explicit} on Problem~3.
For RAIN-SL, NaN marks the first plotted query count at which at least one of the $50$ runs 
becomes nonfinite.

Problems~1-3 are monotone, with $A\neq0$ in Problem~3.
For rock-paper-scissors (RPS), we parameterize the mixed strategies by softmax logits.
The resulting operator
is nonmonotone, so we include this problem
to examine the behavior of the methods outside the monotone setting.
Appendix~\ref{app:experiment-problems} gives the exact problems.

The RRSEG curves decrease most rapidly near their terminal horizons.
RAIN-SL decreases rapidly at early query counts, but at least one of its $50$ runs becomes nonfinite on every problem.
For RPS, softmax saturation can make the residual in the logit variables small without making the saddle-point gap small.
Appendix~\ref{app:rps-gap} therefore also reports the mixed-strategy saddle-point gap,
which decreases together with the squared residual for RRSEG and both VRAF variants.

%% file: appendices/further_related_work.tex

\paragraph{Gap lower bounds.}
First-order lower complexity bounds are known for convex-concave bilinear saddle-point problems under saddle-point criteria \citep{ouyang2021lower}.
For smooth convex-concave saddle-point problems, \citet{golowich2020last} prove an $O(N^{-1/2})$ last-iterate rate for extragradient and matching lower bounds for stationary linear iterative methods, for criteria including the restricted primal-dual gap and the Hamiltonian.
These criteria differ from the tangent residual used in this paper, so these results do not imply the lower bound in Section~\ref{sec:acceleration-barrier}.

\paragraph{Beyond monotonicity.}
Several first-order methods retain convergence guarantees under assumptions weaker than monotonicity.
EG+ is analyzed for structured nonconvex-nonconcave problems under a weak Minty-type condition \citep{diakonikolas2021efficient}, and \citet{pethick2022escaping} extend extragradient-type methods to constrained and regularized weak-Minty problems.
FEG attains accelerated residual convergence under negative comonotonicity \citep{lee2021fast}, while proximal-point, extragradient, and optimistic methods have also been analyzed under negative comonotonicity \citep{gorbunov2023convergence}.
Optimistic dual extrapolation and semi-anchored schemes provide further guarantees for coherent nonmonotone variational inequalities and structured nonconvex-nonconcave composite minimax problems \citep{song2020optimistic,lee2022semianchored}.

\paragraph{Implicit and strong-convergence viewpoints.}
Anchoring is closely connected to Halpern iteration and proximal-point methods.
Halpern-based methods give near-optimal guarantees for monotone inclusions \citep{diakonikolas2020halpern}, while accelerated proximal-point and Halpern-type splitting methods give fast residual guarantees for maximally monotone operators and monotone inclusions \citep{kim2021accelerated,tran-dinh2021halperntype}.
For nonexpansive and contractive fixed-point problems, \citet{park2022exact} derive exact optimal accelerated methods.
Flexible anchoring has also been used to obtain strong convergence together with fast residual decay \citep{bot2026extragradient}, and \citet{yoon2025accelerated} identify common trajectory structure among several accelerated minimax and fixed-point methods.

\paragraph{Continuous-time and structural interpretations of anchoring.}
Continuous-time models have been used to explain optimism and anchoring in minimax dynamics \citep{ryu2020ode} and to characterize the rate produced by general anchor schedules \citep{suh2023continuoustime}.
Flexible anchored extragradient can be derived by discretizing a Tikhonov-regularized monotone flow \citep{bot2026extragradient}, while operator-side Tikhonov regularization gives a common construction for Halpern, forward, extragradient, and Popov-type methods \citep{chen2026unifying}.

\paragraph{Variance-reduction mechanisms.}
Recursive stochastic estimators have a broad literature outside monotone problems.
STORM and PAGE are representative recursive estimators for stochastic nonconvex optimization \citep{cutkosky2019momentumbased,li2021page}, and ROOT-SGD uses recursive averaging under a stochastic Lipschitz condition on the gradient noise \citep{li2022rootsgd}.
For variational inequalities and monotone inclusions, variance reduction has been developed directly for extragradient and splitting schemes \citep{alacaoglu2022stochastic} and for Halpern iteration \citep{cai2022stochastic}.

%% file: appendices/rain.tex
Table~\ref{tab:stochastic-methods} uses the criterion $\E\|F(z)\|^2\le\epsilon$, while the RAIN theorems of \citet{chen2024nearoptimal} are stated for $\E\|F(z)\|$. The required second-moment rate follows from the same proof.

\begin{proposition}\label{prop:rain-squared-residual}
For a monotone $L$-Lipschitz equation with $\|z_0-z_\star\|\le D$, RAIN attains
\begin{align*}
 \E\|F(z_S)\|^2\le\epsilon
\end{align*}
with $\widetilde O(LD/\sqrt\epsilon+\sigma^2/\epsilon)$ oracle queries.
In particular, the stochastic contribution in the oracle bound is
$O\left(
\frac{\sigma^2}{\epsilon}
\log^3\!\left(1+\frac{LD}{\sqrt{\epsilon}}\right)
\right)$.
\end{proposition}

\begin{proof}
First consider the strongly monotone subproblem used by RAIN. Lemma~3.1 and equation~(17) in the proof of Theorem~4.1 of \citet{chen2024nearoptimal} give, for a target $\delta>0$,
\begin{align*}
 \|G(z_S)\|\le16\sum_{s=0}^{S-1}\lambda_s\|e_s\|,
 \qquad
 256\lambda_s^2S^2\E\|e_s\|^2\le\delta^2.
\end{align*}
Hence Cauchy-Schwarz gives $\E\|G(z_S)\|^2\le\delta^2$. For a general monotone equation, RAIN applies this result to $G(z)=F(z)+\lambda(z-z_0)$. Lemma~3.2 of \citet{chen2024nearoptimal} gives $\|F(z)\|\le2\|G(z)\|+\lambda D$. Choosing the internal tolerance and $\lambda D$ as constant multiples of $\sqrt\epsilon$ therefore yields $\E\|F(z_S)\|^2\le\epsilon$, with only constant changes to the parameters. Substituting $\delta=\Theta(\sqrt\epsilon)$ and $\lambda=\Theta(\sqrt\epsilon/D)$ into the oracle count in Theorems~4.1-4.2 of \citet{chen2024nearoptimal} gives the stated order.
\end{proof}

%% file: appendices/vraf_full_proof.tex
\subsection{Potential and residual extraction}

Let $\mathcal F_k:=\sigma(\xi_0,\ldots,\xi_k)$ and
$\E_k[\cdot]:=\E[\cdot\mid\mathcal F_k]$.
For $k\ge1$, let $h_k\in A(z_k)$ be the element selected by the resolvent at iteration $k-1$, so that
\begin{align*}
 z_k+\alpha_{k-1}h_k
 &= (1-\beta_{k-1})z_{k-1}+\beta_{k-1}z_0-\alpha_{k-1}v_{k-1}.
\end{align*}
Define
\begin{align*}
 T_k &:= \frac{\alpha_k}{\beta_k},
 &
 u_k &:= T_k\bigl(v_k-F(z_k)\bigr) \quad (k\ge0),\\
 p_k &:= T_{k-1}\bigl(F(z_k)+h_k\bigr) \quad (k\ge1),
 &
 d &:= z_0-z_\star,
\end{align*}
and, for $k\ge1$,
\begin{align*}
 V_k
 := \frac12\E\|z_k-z_0+u_k+p_k\|^2
 + \frac12\E\|u_k-d\|^2.
\end{align*}

Since $p_k/T_{k-1}=F(z_k)+h_k\in(F+A)(z_k)$ and
$0\in(F+A)(z_\star)$, monotonicity gives
\begin{align*}
 \|z_k-z_\star+p_k\|^2-\|p_k\|^2
 &= \|z_k-z_\star\|^2+2\langle p_k,z_k-z_\star\rangle\ge0.
\end{align*}
Hence
\begin{align}
 \E[\mathcal R(z_k)^2]
 &\le \frac1{T_{k-1}^2}\E\|p_k\|^2\notag\\
 &\le \frac1{T_{k-1}^2}\E\|z_k-z_\star+p_k\|^2\notag\\
 &= \frac1{T_{k-1}^2}\E\bigl\| (z_k-z_0+u_k+p_k)-(u_k-d)\bigr\|^2\notag\\
 &\le \frac{4V_k}{T_{k-1}^2}.
 \label{vh:eq:residual-extraction}
\end{align}
We first prove the uniform bound
\begin{align}
 V_k \le \frac9{10}\left(D^2+\frac{\sigma^2}{\barL^2}\right),
 \qquad k\ge1.
 \label{vh:eq:potential-target}
\end{align}

\subsection{One-step bound}
The exact algebraic identities and coefficient-positivity checks in this proof
are independently verified by the symbolic proof checker provided in the
supplementary material; all checks use exact arithmetic.
\begin{proposition}\label{vh:prop:one-step}
For every $k\ge1$,
\begin{align}
 V_{k+1}
 \le \left(1-\frac{2}{5(k+3)}\right)V_k
 + \frac{2}{5(k+3)}\frac9{10} \left(D^2+\frac{\sigma^2}{\barL^2}\right).
 \label{vh:eq:regular-drift}
\end{align}
\end{proposition}

\begin{proof}
Fix $k\ge1$ and set $t:=k+3\ge4$. Since $\beta_k=3/t$ and
$T_k=\alpha_k/\beta_k$, the schedule gives
\begin{align*}
 \frac{T_{k-1}}{T_k}
 &=\frac{\alpha_{k-1}}{\alpha_k}\frac{\beta_k}{\beta_{k-1}}
 = \frac{2t-1}{2(t-1)}\frac{t-1}{t}
 = \frac{2t-1}{2t},\\
 \frac{T_{k+1}}{T_k}
 &= \frac{\alpha_{k+1}}{\alpha_k}\frac{\beta_k}{\beta_{k+1}}
 = \frac{2t}{2t+1}\frac{t+1}{t}
 = \frac{2(t+1)}{2t+1},\\
 (1-\gamma_{k+1})\frac{T_{k+1}}{T_k}
 &= \left(1-\frac{4t+1}{4t(t+1)}\right) \frac{2(t+1)}{2t+1}
 = \frac{2t-1}{2t}.
\end{align*}
Set
\begin{align*}
 q := \frac{2t-1}{2t},
 \qquad
 \rho := 1-\frac{2}{5t}.
\end{align*}
We will also use
\begin{align}
 \barL^2\alpha_k^2 \le \frac1{k+3},
 \qquad k\ge1.
 \label{vh:eq:alpha-envelope}
\end{align}
Indeed, $\alpha_1=1/(2\barL)$, so the claim holds for $k=1$. If it holds at $k$, then
\begin{align*}
 \barL^2\alpha_{k+1}^2
 &\le \frac1t\left(\frac{2t}{2t+1}\right)^2
 = \frac{4t}{(2t+1)^2}
 < \frac1{t+1}.
\end{align*}

Set
\begin{align*}
 s := \frac{z_k-z_{k+1}}{\beta_k},
 \qquad
 g := T_k\bigl(F(z_{k+1})-F(z_k)\bigr),
 \qquad
 w := z_k-z_0+u_k+p_k.
\end{align*}
The resolvent relation at iteration $k$ gives
\begin{align*}
 z_{k+1}+\alpha_kh_{k+1}
 &= (1-\beta_k)z_k+\beta_kz_0-\alpha_kv_k,\\
 s
 &= z_k-z_0+T_k(v_k+h_{k+1})\\
 &= z_k-z_0+u_k+p_{k+1}-g.
\end{align*}

Subtracting $F(z_{k+1})$ from the definition of $v_{k+1}$ and multiplying by $T_{k+1}$ gives
\begin{align}
 u_{k+1}
 =& qu_k+qT_k\bigl(
 \delta_{\xi_{k+1}}(z_{k+1})-\delta_{\xi_{k+1}}(z_k)\bigr)\notag\\
 &\quad +(T_{k+1}-qT_k)\delta_{\xi_{k+1}}(z_{k+1})\notag\\
 =& qu_k+\omega_{k+1},
 \label{vh:eq:u-step}
\end{align}
where
\begin{align*}
 \omega_{k+1}
 :=& qT_k\bigl( \delta_{\xi_{k+1}}(z_{k+1})-\delta_{\xi_{k+1}}(z_k)\bigr)
 + (T_{k+1}-qT_k)\delta_{\xi_{k+1}}(z_{k+1}).
\end{align*}
Both $z_k$ and $z_{k+1}$ are $\mathcal F_k$-measurable, so unbiasedness gives
\begin{align*}
 \E_k\omega_{k+1}=0.
\end{align*}
Since $\E u_0=T_0\E\delta_{\xi_0}(z_0)=0$, \eqref{vh:eq:u-step} also gives
\begin{align}
 \E u_k=0,
 \qquad
 \E\langle u_k,d\rangle=0,
 \qquad k\ge0.
 \label{vh:eq:u-centered}
\end{align}
Young's inequality with $\varepsilon=9/20$, together with Assumptions~\ref{ass:stochastic-oracle} and~\ref{ass:noise-lipschitz}, gives
\begin{align}
 \E_k\|\omega_{k+1}\|^2
 &\le \frac{29}{20}q^2T_k^2L_\Delta^2 \|z_{k+1}-z_k\|^2
 + \frac{29}{9}(T_{k+1}-qT_k)^2\sigma^2\notag\\
 &= \frac{29}{20}q^2\alpha_k^2L_\Delta^2\|s\|^2
 + \frac{29}{9}(T_{k+1}-qT_k)^2\sigma^2.
 \label{vh:eq:noise-bound}
\end{align}

Using $z_{k+1}=z_k-\beta_ks$, \eqref{vh:eq:u-step}, and
$z_k-z_0+p_{k+1}=s-u_k+g$,
\begin{align*}
 z_{k+1}-z_0+u_{k+1}+p_{k+1}
 &= (1-\beta_k)s+(q-1)u_k+g+\omega_{k+1},\\
 u_{k+1}-d
 &= qu_k-d+\omega_{k+1}.
\end{align*}
Substituting these identities into $V_{k+1}$, using
$\E_k\omega_{k+1}=0$, and subtracting $\rho V_k$ gives
\begin{align}
 V_{k+1}-\rho V_k
 =& \frac12\E\|(1-\beta_k)s+(q-1)u_k+g\|^2
 -\frac\rho2\E\|w\|^2\notag\\
 & +\frac{q^2-\rho}{2}\E\|u_k\|^2
 + \frac{1-\rho}{2}\|d\|^2
 + \E\|\omega_{k+1}\|^2,
 \label{vh:eq:potential-difference}
\end{align}
where \eqref{vh:eq:u-centered} removes the cross term with $d$.

Monotonicity of $F+A$ between $z_k$ and $z_{k+1}$ gives
\begin{align*}
 0
 &\le \langle F(z_k)+h_k-F(z_{k+1})-h_{k+1},z_k-z_{k+1}\rangle\\
 &= \frac{\beta_k}{qT_k}\langle p_k-qp_{k+1},s\rangle.
\end{align*}
Since $\beta_k/(qT_k)>0$ and
$p_k-qp_{k+1}=w-s-g+(1-q)p_{k+1}$,
\begin{align}
 0
 &\le \langle w,s\rangle+(1-q)\langle p_{k+1},s\rangle
 -\|s\|^2-\langle g,s\rangle.
 \label{vh:eq:step-monotonicity-1}
\end{align}
Monotonicity between $z_{k+1}$ and $z_\star$ gives
\begin{align}
 0
 &\le T_k\langle F(z_{k+1})+h_{k+1},z_{k+1}-z_\star\rangle\notag\\
 &= \langle p_{k+1},d-u_k-p_{k+1}+g+(1-\beta_k)s\rangle.
 \label{vh:eq:step-monotonicity-2}
\end{align}
Finally, Lipschitz continuity of $F$ gives
\begin{align}
 0
 &\le \alpha_k^2L^2\|s\|^2-\|g\|^2.
 \label{vh:eq:step-lipschitz}
\end{align}

For $t\ge4$, the coefficients
$1-8/(5t)$, $1/(5t)$, and $29q^2/20$ are nonnegative. Add these multiples of
\eqref{vh:eq:step-monotonicity-1}, \eqref{vh:eq:step-monotonicity-2}, and
\eqref{vh:eq:step-lipschitz}, respectively, to the right-hand side of
\eqref{vh:eq:potential-difference}. Then \eqref{vh:eq:noise-bound} gives
\begin{align*}
 V_{k+1}-\rho V_k
 \le& \E\Bigg[
 \frac12\left\|\left(1-\frac3t\right)s+(q-1)u_k+g\right\|^2
 -\frac\rho2\|w\|^2
 +\frac{q^2-\rho}{2}\|u_k\|^2\\
 & +\left(1-\frac8{5t}\right)
 \bigl(\langle w,s\rangle+(1-q)\langle p_{k+1},s\rangle
       -\|s\|^2-\langle g,s\rangle\bigr)\\
 & +\frac1{5t}\langle p_{k+1},d-u_k-p_{k+1}+g
                  +(1-3/t)s\rangle\\
 & +\frac{29}{20}q^2\alpha_k^2\barL^2\|s\|^2
 -\frac{29}{20}q^2\|g\|^2\Bigg]\\
 & +\frac{1-\rho}{2}D^2
 +\frac{29}{9}(T_{k+1}-qT_k)^2\sigma^2.
\end{align*}
Using \eqref{vh:eq:alpha-envelope} and
\begin{align*}
 \left(1-\frac8{5t}\right)(1-q)
 + \frac{1-3/t}{5t}
 = \frac{7(t-2)}{10t^2},
\end{align*}
this becomes
\begin{align*}
 V_{k+1}-\rho V_k
 \le& \E\Bigg[
 \frac12\left\|\left(1-\frac3t\right)s+(q-1)u_k+g\right\|^2
 +\frac{q^2-\rho}{2}\|u_k\|^2\\
 & -\frac\rho2\|w\|^2
 +\left(1-\frac8{5t}\right)\langle w,s\rangle\\
 & -\frac1{5t}\|p_{k+1}\|^2
 +\frac1{5t}\left\langle p_{k+1},
 d-u_k+g+\frac{7(t-2)}{2t}s\right\rangle\\
 & +\left(\frac{29q^2}{20t}-1+\frac8{5t}\right)\|s\|^2
 -\left(1-\frac8{5t}\right)\langle g,s\rangle
 -\frac{29}{20}q^2\|g\|^2\Bigg]\\
 & +\frac{1-\rho}{2}D^2
 +\frac{29}{9}(T_{k+1}-qT_k)^2\sigma^2.
\end{align*}
The two terms containing $w$ satisfy
\begin{align*}
 -\frac\rho2\|w\|^2
 +\left(1-\frac8{5t}\right)\langle w,s\rangle
 &= -\frac\rho2\left\|w-
 \frac{1-8/(5t)}{\rho}s\right\|^2
 +\frac{(1-8/(5t))^2}{2\rho}\|s\|^2,
\end{align*}
and the two terms containing $p_{k+1}$ satisfy
\begin{align*}
 & -\frac1{5t}\|p_{k+1}\|^2
 +\frac1{5t}\left\langle p_{k+1},
 d-u_k+g+\frac{7(t-2)}{2t}s\right\rangle\\
 &\quad=-\frac1{5t}\left\|p_{k+1}-\frac12(d-u_k+g)
 -\frac{7(t-2)}{4t}s\right\|^2
 +\frac1{20t}\left\|d-u_k+g+\frac{7(t-2)}{2t}s\right\|^2.
\end{align*}
After discarding these two nonpositive squared norms, define
\begin{align*}
 \widehat Q_t(u,s,g,d)
 :=& \frac12\left\|\left(1-\frac3t\right)s+(q-1)u+g\right\|^2\\
 & +\left[
 \frac{(1-8/(5t))^2}{2\rho}-1+\frac8{5t}
 +\frac{29q^2}{20t}\right]\|s\|^2
 -\left(1-\frac8{5t}\right)\langle g,s\rangle\\
 & +\frac1{20t}\left\|d-u+g+\frac{7(t-2)}{2t}s\right\|^2
 -\frac{29}{20}q^2\|g\|^2
 +\frac{q^2-\rho}{2}\|u\|^2.
\end{align*}
Then
\begin{align}
 V_{k+1}-\rho V_k
 \le \E\widehat Q_t(u_k,s,g,d)
 +\frac{1-\rho}{2}D^2
 +\frac{29}{9}(T_{k+1}-qT_k)^2\sigma^2.
 \label{vh:eq:one-step-reduction}
\end{align}

Substituting $q=(2t-1)/(2t)$ and $\rho=1-2/(5t)$ gives
\begin{align*}
 & \frac25(1-\rho)\|d\|^2
 -\frac{3}{10t}\langle u,d\rangle
 -\widehat Q_t(u,s,g,d)\\
 & =\frac{t-1}{4t^2}\|u\|^2
 +\frac{295t^3-646t^2-1029t+450}
 {80t^3(5t-2)}\|s\|^2
 +\frac{76t^2-120t+29}{80t^2}\|g\|^2
 +\frac{11}{100t}\|d\|^2\\
 &\quad +\frac{17t-44}{20t^2}\langle u,s\rangle
 +\frac3{5t}\langle u,g\rangle
 -\frac1{5t}\langle u,d\rangle
 +\frac{7(3t+2)}{20t^2}\langle s,g\rangle\\
 &\quad -\frac{7(t-2)}{20t^2}\langle s,d\rangle
 -\frac1{10t}\langle g,d\rangle.
\end{align*}
Completing the terms involving $d$ gives
\begin{align*}
 & \frac{11}{100t}
 \left\|d-\frac{10}{11}u-\frac5{11}g
 -\frac{35(t-2)}{22t}s\right\|^2\\
 &\quad +\frac{7t-11}{44t^2}\|u\|^2
 +\frac{2020t^3-1716t^2-18179t+6910}
 {880t^3(5t-2)}\|s\|^2
 +\frac{836t^2-1340t+319}{880t^2}\|g\|^2\\
 &\quad +\frac{117t-344}{220t^2}\langle u,s\rangle
 +\frac{28}{55t}\langle u,g\rangle
 +\frac{7(7t+8)}{55t^2}\langle s,g\rangle.
\end{align*}
Completing the terms involving $u$ gives
\begin{align*}
 & \frac{11}{100t}
 \left\|d-\frac{10}{11}u-\frac5{11}g
 -\frac{35(t-2)}{22t}s\right\|^2\\
 &\quad +\frac{7t-11}{44t^2}
 \left\|u+\frac{(117t-344)s+112tg}{70t-110}\right\|^2\\
 &\quad +\frac{205t^4+23518t^3-117687t^2+134397t-34550}
 {400t^3(5t-2)(7t-11)}\|s\|^2\\
 &\quad +\frac{7(t^2+53t-40)}{25t^2(7t-11)}\langle s,g\rangle
 +\frac{2660t^3-9584t^2+7715t-1595}
 {400t^2(7t-11)}\|g\|^2.
\end{align*}
For $c>0$,
\begin{align*}
 a\|s\|^2+b\langle s,g\rangle+c\|g\|^2
 =c\left\|g+\frac{b}{2c}s\right\|^2
 +\left(a-\frac{b^2}{4c}\right)\|s\|^2.
\end{align*}
Set $m:=t-4\ge0$. For the last three terms above, the coefficient of
$\|g\|^2$ is
\begin{align*}
 \frac{2660m^3+22336m^2+58723m+46161}
 {400(m+4)^2(7m+17)}>0.
\end{align*}
After completing the square in $g$, the remaining coefficient of $\|s\|^2$ is
\begin{align*}
 \frac{
 \left(\begin{aligned}
 &15580m^6+2129188m^5+26216605m^4+125305406m^3\\
 &\qquad+267420385m^2+214441225m+2604726
 \end{aligned}\right)
 }
 {80(m+4)^3(5m+18)
  (2660m^3+22336m^2+58723m+46161)}>0.
\end{align*}
Therefore
\begin{align}
 \widehat Q_t(u,s,g,d)
 &\le \frac25(1-\rho)\|d\|^2
 -\frac{3}{10t}\langle u,d\rangle.
 \label{vh:eq:qhat-bound}
\end{align}
Taking expectation in \eqref{vh:eq:qhat-bound} and using \eqref{vh:eq:u-centered},
\begin{align}
 \E\widehat Q_t(u_k,s,g,d)
 \le \frac25(1-\rho)D^2.
 \label{vh:eq:qhat-expectation}
\end{align}

Finally, $qT_k=(1-\gamma_{k+1})T_{k+1}$, so
$T_{k+1}-qT_k=\gamma_{k+1}T_{k+1}$. By
\eqref{vh:eq:alpha-envelope} at $k+1$ and $\beta_{k+1}=3/(t+1)$,
\begin{align*}
 \barL^2T_{k+1}^2
 = \barL^2\alpha_{k+1}^2\frac{(t+1)^2}{9}
 \le\frac{t+1}{9}.
\end{align*}
Hence
\begin{align*}
 \frac{29}{9}\barL^2(T_{k+1}-qT_k)^2
 &\le \frac{29(4t+1)^2}{1296t^2(t+1)},\\
 \frac9{25t}
 -\frac{29(4t+1)^2}{1296t^2(t+1)}
 &= \frac{64t^2+5864t-725}{32400t^2(t+1)}>0.
\end{align*}
Since $9/(25t)=(9/10)(1-\rho)$,
\begin{align}
 \frac{29}{9}(T_{k+1}-qT_k)^2\sigma^2
 \le\frac9{10}(1-\rho)\frac{\sigma^2}{\barL^2}.
 \label{vh:eq:step-noise-final}
\end{align}
Combining \eqref{vh:eq:one-step-reduction},
\eqref{vh:eq:qhat-expectation}, and \eqref{vh:eq:step-noise-final} gives
\begin{align*}
 V_{k+1}-\rho V_k
 &\le \left(\frac12+\frac25\right)(1-\rho)D^2
 +\frac9{10}(1-\rho)\frac{\sigma^2}{\barL^2}\\
 &= \frac9{10}(1-\rho)
 \left(D^2+\frac{\sigma^2}{\barL^2}\right).
\end{align*}
Since $1-\rho=2/[5(k+3)]$, this is \eqref{vh:eq:regular-drift}.
\end{proof}

\subsection{Initialization and conclusion}

At $k=0$,
\begin{align*}
 \beta_0 &= 1,
 &
 T_0 &= \alpha_0=\frac7{12\barL},
 &
 \alpha_1 &= \frac1{2\barL},\\
 \beta_1 &= \frac34,
 &
 T_1 &= \frac2{3\barL},
 &
 \gamma_1 &= \frac{13}{48}.
\end{align*}
Set
\begin{align*}
 s_0 := z_0-z_1,
 \qquad
 g_0 := T_0\bigl(F(z_1)-F(z_0)\bigr),
 \qquad
 \ell := \frac{L^2}{\barL^2}\in[0,1].
\end{align*}
Since $\beta_0=1$, the resolvent relation gives
\begin{align*}
 s_0
 &= T_0(v_0+h_1)
 = u_0+T_0\bigl(F(z_0)+h_1\bigr)
 = u_0+p_1-g_0,\\
 p_1 &= s_0-u_0+g_0.
\end{align*}
Subtracting $F(z_1)$ from the definition of $v_1$ and multiplying by $T_1$ gives
\begin{align*}
 u_1
 =& \frac56u_0
 +\frac{35}{72\barL}\bigl(
 \delta_{\xi_1}(z_1)-\delta_{\xi_1}(z_0)\bigr)
 +\frac{13}{72\barL}\delta_{\xi_1}(z_1)\\
 =& \frac56u_0+\omega_1,
\end{align*}
where
\begin{align*}
 \omega_1
 := \frac{35}{72\barL}\bigl(
 \delta_{\xi_1}(z_1)-\delta_{\xi_1}(z_0)\bigr)
 +\frac{13}{72\barL}\delta_{\xi_1}(z_1),
 \qquad
 \E_0\omega_1=0.
\end{align*}
Young's inequality with $\varepsilon=3/10$ gives
\begin{align}
 \E_0\|\omega_1\|^2
 &\le \frac{3185}{10368}(1-\ell)\|s_0\|^2
 +\frac{2197}{15552}\frac{\sigma^2}{\barL^2}.
 \label{vh:eq:init-noise}
\end{align}
Since $z_1-z_0+p_1=-u_0+g_0$,
\begin{align*}
 z_1-z_0+u_1+p_1
 &= g_0-\frac16u_0+\omega_1,\\
 u_1-d
 &= \frac56u_0-d+\omega_1.
\end{align*}
Using $\E_0\omega_1=0$ and \eqref{vh:eq:init-noise} in the definition of $V_1$,
\begin{align}
 V_1
 \le& \E\left[
 \frac12\left\|g_0-\frac16u_0\right\|^2
 + \frac12\left\|\frac56u_0-d\right\|^2
 + \frac{3185}{10368}(1-\ell)\|s_0\|^2\right]
 + \frac{2197}{15552}\frac{\sigma^2}{\barL^2}.
 \label{vh:eq:init-prebound}
\end{align}

Monotonicity, Lipschitz continuity, and unbiasedness give
\begin{align}
 \langle s_0-u_0+g_0,d-s_0\rangle &\ge 0,
 \label{vh:eq:init-monotonicity}\\
 \frac{49\ell}{144}\|s_0\|^2-\|g_0\|^2 &\ge 0,
 \label{vh:eq:init-lipschitz}\\
 \E\langle u_0,d\rangle &= 0.
 \label{vh:eq:init-centered}
\end{align}
For any $g,u,d,s\in\mathcal H$,
\begin{align}
 & \frac9{10}\|d\|^2+\frac{21}{10}\|u\|^2
 -\frac12\left\|g-\frac16u\right\|^2
 -\frac12\left\|\frac56u-d\right\|^2
 -\frac{3185}{10368}(1-\ell)\|s\|^2\notag\\
 &\quad -\frac7{10}\langle s-u+g,d-s\rangle
 -\left(\frac{49\ell}{144}\|s\|^2-\|g\|^2\right)
 -\frac45\langle u,d\rangle\notag\\
 &= \frac12\left\|g+\frac16u-\frac7{10}d+\frac7{10}s\right\|^2
 +\frac{69}{40}\left\|u+\frac{17}{69}d-\frac{49}{207}s\right\|^2\notag\\
 &\quad +\frac{347}{6900}\left\|d-\frac{91}{1041}s\right\|^2
 +\frac{4417}{249840}\|s\|^2
 +(1-\ell)\frac{343}{10368}\|s\|^2.
 \label{vh:eq:init-factorization}
\end{align}
The right-hand side is nonnegative. Applying \eqref{vh:eq:init-factorization} to
$(g_0,u_0,d,s_0)$, taking expectation, and using
\eqref{vh:eq:init-monotonicity}--\eqref{vh:eq:init-centered} gives
\begin{align}
 & \E\left[
 \frac12\left\|g_0-\frac16u_0\right\|^2
 +\frac12\left\|\frac56u_0-d\right\|^2
 +\frac{3185}{10368}(1-\ell)\|s_0\|^2\right]\notag\\
 &\qquad \le\frac9{10}\|d\|^2+\frac{21}{10}\E\|u_0\|^2.
 \label{vh:eq:init-quadratic-bound}
\end{align}
Moreover,
\begin{align*}
 \E\|u_0\|^2
 =T_0^2\E\|\delta_{\xi_0}(z_0)\|^2
 \le \frac{49}{144}\frac{\sigma^2}{\barL^2}.
\end{align*}
Combining this bound with \eqref{vh:eq:init-prebound} and
\eqref{vh:eq:init-quadratic-bound},
\begin{align}
 V_1
 &\le \frac9{10}\|d\|^2
 +\frac{66551}{77760}\frac{\sigma^2}{\barL^2}\notag\\
 &\le \frac9{10}\left(D^2+\frac{\sigma^2}{\barL^2}\right).
 \label{vh:eq:init-bound}
\end{align}
By \eqref{vh:eq:regular-drift}, \eqref{vh:eq:init-bound} implies
\eqref{vh:eq:potential-target} by induction.

It remains to lower-bound $T_k$. The schedule gives
\begin{align*}
 \frac{T_{k+1}}{T_k}
 &= \frac{2(k+4)}{2k+7},\\
 \frac{T_{k+1}^2/(k+4)}{T_k^2/(k+3)}
 &= 1-\frac1{(2k+7)^2}.
\end{align*}
Since $\barL^2T_0^2/3=49/432$,
\begin{align*}
 \frac{\barL^2T_k^2}{k+3}
 &= \frac{49}{432}
 \prod_{j=0}^{k-1}\left(1-\frac1{(2j+7)^2}\right)\\
 &\ge \frac{49}{432}
 \prod_{j=0}^{\infty}\left(1-\frac1{(2j+7)^2}\right)\\
 &= \frac{49}{432}\frac{75\pi}{256}
 = \frac{1225\pi}{36864}
 > \frac{18}{175},
\end{align*}
where the equality uses the Wallis product and the last inequality follows from $\pi>31/10$. Thus
\begin{align}
 T_k^2
 &\ge \frac{1225\pi}{36864\barL^2}(k+3)
 > \frac{18}{175\barL^2}(k+3),
 \qquad k\ge0.
 \label{vh:eq:scale-lower}
\end{align}
Finally, \eqref{vh:eq:residual-extraction}, \eqref{vh:eq:potential-target}, and
\eqref{vh:eq:scale-lower} give, for every $N\ge1$,
\begin{align*}
 \E[\mathcal R(z_N)^2]
 &\le \frac{4V_N}{T_{N-1}^2}\\
 &\le 4\cdot\frac9{10}
 \left(D^2+\frac{\sigma^2}{\barL^2}\right)
 \frac{175\barL^2}{18(N+2)}\\
 &= 35\frac{\barL^2D^2+\sigma^2}{N+2}.
\end{align*}
This proves Theorem~\ref{thm:inclusion}.

\subsection{Tightness of the deterministic rate}\label{app:vraf-deterministic-tightness}
The $O(1/N)$ deterministic term in Theorem~\ref{thm:inclusion} is tight for the VRAF schedule.  The following argument applies more generally to the~\eqref{eq:vraf} whenever
$L\alpha_k=\Theta(k^{-1/2})$ and $\beta_k=\Theta(k^{-1})$.

\begin{lemma}\label{lem:vraf-deterministic-tightness}
Consider the deterministic iteration
\begin{align}
 z_{k+1}=z_k-\alpha_kF(z_k)+\beta_k(z_0-z_k),
 \label{eq:anchored-forward-tightness}
\end{align}
where $\alpha_k,\beta_k>0$, $L\alpha_k=\Theta(k^{-1/2})$, and
$\beta_k=\Theta(k^{-1})$.  For the one-dimensional monotone $L$-Lipschitz operator
$F(z)=Lz$ and $z_0=D>0$,
\begin{align*}
 \|F(z_k)\|^2=\Omega\left(\frac{L^2D^2}{k}\right).
\end{align*}
\end{lemma}

\begin{proof}
For $F(z)=Lz$,
\begin{align}
 z_{k+1}=(1-L\alpha_k-\beta_k)z_k+D\beta_k.
 \label{eq:linear-tightness-recursion}
\end{align}
For all sufficiently large $k$,
$0\le 1-L\alpha_k-\beta_k\le 1$, and there are constants $a,b>0$ such that
\begin{align*}
 L\alpha_k\le \frac{a}{\sqrt{k+1}},
 \qquad
 \beta_k\ge \frac{b}{k+1}.
\end{align*}
Moreover, $z_k>0$ from some index onward.  Whenever $z_k\le0$,
\eqref{eq:linear-tightness-recursion} gives
\begin{align*}
 z_{k+1}-z_k
 =-(L\alpha_k+\beta_k)z_k+D\beta_k
 \ge D\beta_k.
\end{align*}
Since $\sum_k\beta_k=\infty$, the iterates cannot remain nonpositive; once
$z_k>0$, the nonnegative coefficient in
\eqref{eq:linear-tightness-recursion} keeps all subsequent iterates positive.

Choose an index $K$ from which these properties hold and choose $c>0$ sufficiently small that
\begin{align*}
 \frac{cD}{\sqrt{K+1}}\le z_K,
 \qquad
 c\le\frac12,
 \qquad
 ac<\frac b2.
\end{align*}
Set $w_k=cD/\sqrt{k+1}$.  For every $k\ge K$,
\begin{align*}
 (1-L\alpha_k-\beta_k)w_k+D\beta_k-w_{k+1}
 &=(w_k-w_{k+1})+\beta_k(D-w_k)-L\alpha_kw_k\\
 &\ge \frac{bD}{2(k+1)}-\frac{acD}{k+1}>0.
\end{align*}
Hence \eqref{eq:linear-tightness-recursion} and induction give
$z_k\ge w_k=cD/\sqrt{k+1}$ for every $k\ge K$.  Therefore
\begin{align*}
 \|F(z_k)\|^2=L^2z_k^2
 \ge \frac{c^2L^2D^2}{k+1},
\end{align*}
which proves the claim.
\end{proof}

In the deterministic case, VRAF has $v_k=F(z_k)$ and therefore reduces to
\eqref{eq:anchored-forward-tightness} when $A=0$.  Moreover,
Remark~\ref{rem:vraf-stepsize-decay} gives $L\alpha_k=\Theta(k^{-1/2})$,
and $\beta_k=3/(k+3)=\Theta(k^{-1})$.  Lemma~\ref{lem:vraf-deterministic-tightness}
therefore shows that the deterministic $O(L^2D^2/N)$ term in
Theorem~\ref{thm:inclusion} is order-tight for the VRAF schedule.

%% file: appendices/acceleration_barrier_proof.tex
We prove the lower bound for adaptive randomized algorithms. At each oracle query, the algorithm may choose the query point and may query either a fresh stochastic operator or any previously sampled stochastic operator, based on its internal randomness and all previous observations.

\subsection{Hard family}\label{lb:sec:family}

We use a rescaled version of the one-dimensional noisy binary search family in \citet[Theorem~11]{foster2019complexity}. Fix $M\ge16$ and a hidden index $j\in\{0,\ldots,M-1\}$. Let $Z=(Z_0,\ldots,Z_M)\in\{-1,1\}^{M+1}$ have independent coordinates with
\begin{align*}
 \Pp_j(Z_i=1)=
 \begin{cases}
 \frac12-\frac1{8M}, & i\le j,\\
 \frac12+\frac1{8M}, & i>j.
 \end{cases}
\end{align*}
Define $F_Z:\R\to\R$ by
\begin{align*}
 F_Z(x)=
 \begin{cases}
 \frac{Z_0}{M}, & x<0,\\
 \frac{(i+1-Mx)Z_i+(Mx-i)Z_{i+1}}{M},
 & x\in[i/M,(i+1)/M],\quad i=0,\ldots,M-1,\\
 \frac{Z_M}{M}, & x>1.
 \end{cases}
\end{align*}
Let $\E_j$ denote expectation under $\Pp_j$ and define $F_j:=\E_j[F_Z]$. For the monotone inclusion, set $A=0$ and $z_0=0$.

\begin{lemma}\label{lb:lem:parameters}
For every $M\ge16$ and $j\in\{0,\ldots,M-1\}$, $F_j$ is monotone and $1/(2M)$-Lipschitz, with unique zero $x_j^\star=(j+\tfrac12)/M$. The oracle $F_Z$ satisfies Assumptions~\ref{ass:stochastic-oracle} and~\ref{ass:noise-lipschitz}. Hence Assumptions~\ref{ass:problem}--\ref{ass:noise-lipschitz} hold with
\begin{align}
 D=1,
 \qquad
 L=\frac1{2M},
 \qquad
 \sigma^2=\frac1{M^2},
 \qquad
 L_\Delta=3.
 \label{lb:eq:hard-parameters}
\end{align}
\end{lemma}

\begin{proof}
The coordinate means are
\begin{align*}
 \E_j[Z_i]=
 \begin{cases}
 -\frac1{4M}, & i\le j,\\
 \frac1{4M}, & i>j.
 \end{cases}
\end{align*}
Therefore
\begin{align*}
 F_j(x)=
 \begin{cases}
 -\frac1{4M^2}, & x\le \frac jM,\\[2mm]
 \frac{x-(j+\tfrac12)/M}{2M},
 & \frac jM\le x\le\frac{j+1}{M},\\[3mm]
 \frac1{4M^2}, & x\ge\frac{j+1}{M}.
 \end{cases}
\end{align*}
Thus $F_j$ is monotone and $1/(2M)$-Lipschitz and vanishes only at $x_j^\star$. Since $x_j^\star\in(0,1)$ and $z_0=0$, $|z_0-x_j^\star|\le1$.

Unbiasedness follows from $F_j=\E_j[F_Z]$. Moreover, $|F_Z(x)|\le1/M$ for every $x$, so
\begin{align*}
 \E_j|F_Z(x)-F_j(x)|^2
 =\operatorname{Var}_{\Pp_j}(F_Z(x))
 \le\frac1{M^2}.
\end{align*}
Finally, adjacent grid values of $F_Z$ differ by at most $2/M$ over an interval of length $1/M$, so every $F_Z$ is $2$-Lipschitz. Since $F_j$ is $1/(2M)$-Lipschitz,
\begin{align*}
 |[F_Z(x)-F_j(x)]-[F_Z(y)-F_j(y)]|
 &\le \left(2+\frac1{2M}\right)|x-y|\\
 &\le 3|x-y|.
\end{align*}
This proves \eqref{lb:eq:hard-parameters}.
\end{proof}

\subsection{Reduction to noisy binary search}\label{lb:sec:reduction}

\begin{lemma}[Adaptive noisy binary search;
cf. {\citet[proof of Theorem~11]{foster2019complexity}}]\label{lb:lem:adaptive-nbs}
Let $J$ be uniform on $\{0,\ldots,M-1\}$ and, conditional on $J=j$, let the rows $Z^{(s)}=(Z_0^{(s)},\ldots,Z_M^{(s)})$, $s\ge1$, be independent with independent coordinates satisfying
\begin{align*}
 \Pp(Z_i^{(s)}=1\mid J=j)=
 \begin{cases}
  \frac12-\frac1{8M}, & i\le j,\\
  \frac12+\frac1{8M}, & i>j.
 \end{cases}
\end{align*}
An adaptive randomized algorithm may query arbitrary entries $(s,i)$, including multiple entries from the same row and repeated entries. If it identifies $J$ with probability at least $3/4$, then it requires $\Omega(M^2\log M)$ matrix-entry queries.
\end{lemma}

\begin{proof}
Let $a:=1/(8M)$ and let $\mathcal T_q$ be the transcript after $q$ queries, including the algorithm's internal randomness. The query location at step $t$ is determined by $\mathcal T_{t-1}$. A repeated query to an already observed entry gives no additional information. For a new entry, conditional on $\mathcal T_{t-1}$ its distribution is either $\operatorname{Ber}(\frac12-a)$ or $\operatorname{Ber}(\frac12+a)$, according to the queried coordinate and $J$. Hence
\begin{align*}
 I(J;Y_t\mid\mathcal T_{t-1})
 &\le \max\!\left\{
 D_{\rm KL}\!\left(\operatorname{Ber}(\tfrac12+a)\,\middle\|\,\operatorname{Ber}(\tfrac12-a)\right),
 D_{\rm KL}\!\left(\operatorname{Ber}(\tfrac12-a)\,\middle\|\,\operatorname{Ber}(\tfrac12+a)\right)
 \right\}\\
 &=2a\log\frac{1+2a}{1-2a}
 \le16a^2
 =\frac1{4M^2}.
\end{align*}
By the chain rule, $I(J;\mathcal T_q)\le q/(4M^2)$. If an estimator $\widehat J$ based on $\mathcal T_q$ has error probability at most $1/4$, Fano's inequality gives
\begin{align*}
 I(J;\mathcal T_q)
 \ge \frac34\log M-\log2
 \ge \frac12\log M,
\end{align*}
where the last inequality uses $M\ge16$. Therefore $q\ge2M^2\log M$, proving the claimed order.
\end{proof}

Each query to a sampled stochastic operator can be simulated using at most two matrix-entry queries. If $x\in[i/M,(i+1)/M]$, then
\begin{align*}
 F_{Z^{(s)}}(x)
 =\frac{(i+1-Mx)Z_i^{(s)}+(Mx-i)Z_{i+1}^{(s)}}{M},
\end{align*}
so only $Z_i^{(s)}$ and $Z_{i+1}^{(s)}$ are needed; outside $[0,1]$, only one boundary entry is needed. Repeatedly querying the same sampled stochastic operator corresponds to querying further entries of the same matrix row. Hence $Q$ stochastic-oracle queries reveal at most $2Q$ matrix entries, and Lemma~\ref{lb:lem:adaptive-nbs} implies that identification with probability at least $3/4$ requires
\begin{align}
 Q=\Omega(M^2\log M).
 \label{lb:eq:foster-lower-bound}
\end{align}

\subsection{Residual accuracy implies identification}\label{lb:sec:identify}

Define
\begin{align*}
 \tau_M:=\frac1{8M^2},
 \qquad
 I_j:=\left[\frac{j+\tfrac14}{M},\frac{j+\tfrac34}{M}\right].
\end{align*}
Since $A=0$, $\mathcal R(x)=|F_j(x)|$.

\begin{lemma}\label{lb:lem:hard-lower-bound}
For every $M\ge16$, any adaptive randomized algorithm satisfying
\begin{align}
 \sup_{0\le j<M}\E_j|F_j(\widehat X)|^2
 \le\frac1{256M^4}
 \label{lb:eq:expected-hard-target}
\end{align}
for the family in Subsection~\ref{lb:sec:family} requires $\Omega(M^2\log M)$ oracle queries.
\end{lemma}

\begin{proof}
The expression for $F_j$ in the proof of Lemma~\ref{lb:lem:parameters} gives
\begin{align}
 |F_j(x)|\le\tau_M
 \quad\Longleftrightarrow\quad
 x\in I_j.
 \label{lb:eq:residual-identification}
\end{align}
The intervals $I_0,\ldots,I_{M-1}$ are pairwise disjoint. By Markov's inequality and \eqref{lb:eq:expected-hard-target},
\begin{align*}
 \Pp_j(\widehat X\notin I_j)
 &=\Pp_j\bigl(|F_j(\widehat X)|>\tau_M\bigr)\\
 &\le\frac{\E_j|F_j(\widehat X)|^2}{\tau_M^2}\\
 &\le\frac14.
\end{align*}
Define $\widehat J=\ell$ when $\widehat X\in I_\ell$, and assign any value when $\widehat X$ lies outside all $I_\ell$. Then
\begin{align*}
 \Pp_j(\widehat J=j)
 \ge \Pp_j(\widehat X\in I_j)
 \ge \frac34.
\end{align*}
Thus $\widehat J$ identifies $j$ with probability at least $3/4$, and \eqref{lb:eq:foster-lower-bound} gives the result.
\end{proof}

\subsection{Proof of Theorem~\ref{lb:thm:main}}

\begin{proof}
Fix $M\ge16$ and define on $\R^2$
$\widetilde F_Z(x,u):=(F_Z(x),3u)$ and
$\widetilde F_j(x,u):=(F_j(x),3u)$, with $\widetilde A=0$ and $\widetilde z_0=(0,0)$. By Lemma~\ref{lb:lem:parameters},
\begin{align*}
 \langle\widetilde F_j(x,u)-\widetilde F_j(y,v),(x-y,u-v)\rangle
 &=(F_j(x)-F_j(y))(x-y)+3|u-v|^2\ge0,\\
 \|\widetilde F_j(x,u)-\widetilde F_j(y,v)\|^2
 &\le\frac1{4M^2}|x-y|^2+9|u-v|^2
 \le9\|(x,u)-(y,v)\|^2.
\end{align*}
Thus $\widetilde F_j$ is monotone, and taking $x=y$ and $u\ne v$ shows that its Lipschitz modulus is exactly $L=3$. Its unique zero is $(x_j^\star,0)$, whose distance from $\widetilde z_0$ is at most $1$.

The lifted noise is $\widetilde\delta_Z(x,u)=(F_Z(x)-F_j(x),0)$, so
$\E_j\|\widetilde\delta_Z(x,u)\|^2\le1/M^2$ and
$\E_j\|\widetilde\delta_Z(x,u)-\widetilde\delta_Z(y,v)\|^2\le9\|(x,u)-(y,v)\|^2$.
Hence the lifted family satisfies the assumptions with
$D=1$, $L=3$, $\sigma^2=1/M^2$, and $L_\Delta=3$, and
$\mathcal R_j(x,u)^2=|F_j(x)|^2+9u^2$.

Each query to $\widetilde F_Z(x,u)$ is simulated by one query to $F_Z(x)$ followed by the deterministic coordinate $3u$, including repeated queries to a previously sampled operator. Therefore, at $\epsilon_M:=1/(256M^4)$, lifted residual accuracy implies the scalar criterion in Lemma~\ref{lb:lem:hard-lower-bound}, and hence requires $\Omega(M^2\log M)$ oracle queries. But
$LD/\sqrt{\epsilon_M}+\sigma^2/\epsilon_M=48M^2+256M^2=304M^2$,
so a uniform $O(LD/\sqrt\epsilon+\sigma^2/\epsilon)$ bound gives only $O(M^2)$ queries, a contradiction for sufficiently large $M$.
\end{proof}

\subsection{Proof of Corollary~\ref{lb:cor:fixed-query}}

\begin{proof}
Fix $0\le p<1$ and suppose that a fixed-query method satisfies, uniformly over the problem class, for some constant $C>0$ and every integer $N\ge2$,
\begin{align}
 \E[\mathcal R(\widehat z_N)^2]
 \le C\left(
 \frac{L^2D^2}{N^2}
 +\frac{\sigma^2\log^p N}{N}
 \right).
 \label{lb:eq:sublog-rate}
\end{align}
By the definition of a fixed-query method, there is a constant $r$ such that each iteration uses at most $r$ oracle queries.
Apply \eqref{lb:eq:sublog-rate} to the lifted hard family in the proof of Theorem~\ref{lb:thm:main}, for which
$L=3$, $D=1$, $\sigma^2=M^{-2}$, and $\epsilon_M=(256M^4)^{-1}$.
Choose
\begin{align*}
 N_M=\left\lceil K M^2(\log M)^p\right\rceil,
\end{align*}
where $K$ is a sufficiently large constant depending only on $C$ and $p$.
Since $\log N_M=O(\log M)$,
\begin{align*}
 C\frac{L^2D^2}{N_M^2}
 &=O\!\left(\frac{1}{M^4(\log M)^{2p}}\right),\\
 C\frac{\sigma^2\log^p N_M}{N_M}
 &=O\!\left(\frac{1}{K M^4}\right).
\end{align*}
Thus, after increasing $K$ if necessary, \eqref{lb:eq:sublog-rate} gives
$\E[\mathcal R(\widehat z_{N_M})^2]\le\epsilon_M$ for all sufficiently large $M$.
The method then uses at most
\begin{align*}
 rN_M=O\!\left(M^2(\log M)^p\right)
 =o(M^2\log M)
\end{align*}
oracle queries
 because $p<1$.
This contradicts the $\Omega(M^2\log M)$ lower bound in \eqref{lb:eq:foster-lower-bound} and Lemma~\ref{lb:lem:hard-lower-bound}.
\end{proof}

%% file: appendices/rrseg_fbf_proof.tex
\subsection{Preliminary bounds}

For each $k$, define $H_k(z):=F(z)+a_k(z-c_k)$ and let $r_k$ be the unique zero of $A+H_k$.
The operator $H_k$ is $a_k$-strongly monotone and $(L+a_k)$-Lipschitz.
We first derive bounds that do not depend on the particular horizon-dependent schedule.

\subsubsection{One-step stochastic forward-backward-forward bound}

\begin{lemma}[One-step stochastic forward-backward-forward contraction]
\label{rr:lem:fbf-clean}
Let $B:\cH\to\cH$ be $\mu$-strongly monotone and $M$-Lipschitz, let $A$ be maximally monotone, and let $r$ satisfy $0\in B(r)+A(r)$.
Assume $2\mu\le M$ and consider
\begin{align*}
 y&=J_{\eta A}\bigl(x-\eta(B(x)+\varepsilon)\bigr),\\
 x^+&=y-\eta\bigl(B(y)+\zeta-B(x)-\varepsilon\bigr),
\end{align*}
where $\E\varepsilon=0$, $\E\|\varepsilon\|^2\le\sigma^2$, and, conditional on $\varepsilon$,
\begin{align*}
 \E[\zeta\mid\varepsilon]=0,
 \qquad
 \E[\|\zeta\|^2\mid\varepsilon]\le\sigma^2.
\end{align*}
If $\eta M\le1/3$, then
\begin{align}
 \E\|x^+-r\|^2
 \le\left(1-\frac43\eta\mu\right)\|x-r\|^2
 +\frac52\eta^2\sigma^2.
 \label{rr:eq:fbf-clean}
\end{align}
\end{lemma}

\begin{proof}
The resolvent step gives
\begin{align*}
 h_y:=\frac{x-y}{\eta}-B(x)-\varepsilon\in A(y).
\end{align*}
Choose $h_r=-B(r)\in A(r)$.
Then
\begin{align*}
 x^+=x-\eta\bigl(h_y+B(y)+\zeta\bigr).
\end{align*}
Conditioning on $\varepsilon$ and using the conditional mean of $\zeta$,
\begin{align*}
 \E[\|x^+-r\|^2\mid\varepsilon]
 =&\|x-r\|^2-2\eta\langle h_y+B(y),x-r\rangle\\
 &+\eta^2\|h_y+B(y)\|^2+\eta^2\E[\|\zeta\|^2\mid\varepsilon]\\
 \le&\|x-r\|^2-2\eta\mu\|y-r\|^2
 -2\eta\langle h_y+B(y),x-y\rangle\\
 &+\eta^2\|h_y+B(y)\|^2+\eta^2\sigma^2,
\end{align*}
where the inequality uses monotonicity of $A$ and strong monotonicity of $B$.
Moreover,
\begin{align*}
 &-2\eta\langle h_y+B(y),x-y\rangle
 +\eta^2\|h_y+B(y)\|^2\\
 &\qquad=-\|x-y\|^2+\eta^2\|\varepsilon+B(x)-B(y)\|^2.
\end{align*}
Let $t:=\eta M$.
Young's inequality and Lipschitz continuity of $B$ give
\begin{align*}
 \eta^2\|\varepsilon+B(x)-B(y)\|^2
 &\le\frac{\eta^2}{1-t}\|\varepsilon\|^2
 +\frac{\eta^2}{t}\|B(x)-B(y)\|^2\\
 &\le\frac{\eta^2}{1-t}\|\varepsilon\|^2+t\|x-y\|^2.
\end{align*}
Taking expectation over $\varepsilon$ yields
\begin{align}
 \E\|x^+-r\|^2
 \le&\|x-r\|^2-2\eta\mu\E\|y-r\|^2
 -(1-t)\E\|x-y\|^2\notag\\
 &+\frac{2-t}{1-t}\eta^2\sigma^2.
 \label{rr:eq:fbf-before-contraction}
\end{align}
Write $s:=\eta\mu$.
Since $2\mu\le M$, we have $2s\le t$.
Together with $t\le1/3$, this gives $1-t-4s\ge1-3t\ge0$.
Hence
\begin{align*}
 &2s\|y-r\|^2+(1-t)\|x-y\|^2-\frac43s\|x-r\|^2\\
 &\quad=(1-t+2s)
 \left\|x-y-\frac{2s}{1-t+2s}(x-r)\right\|^2\\
 &\qquad+\frac{2s(1-t-4s)}{3(1-t+2s)}\|x-r\|^2
 \ge0.
\end{align*}
Also,
\begin{align*}
 \frac52-\frac{2-t}{1-t}
 =\frac{1-3t}{2(1-t)}\ge0.
\end{align*}
Substituting the last two inequalities into \eqref{rr:eq:fbf-before-contraction} proves \eqref{rr:eq:fbf-clean}.
\end{proof}

For the schedules considered below, $a_k\le L$, so $2a_k\le L+a_k$.
Applying Lemma~\ref{rr:lem:fbf-clean} conditionally with
\begin{align*}
 B(z)=H_k(z),
 \qquad
 \mu=a_k,
 \qquad
 M=L+a_k,
\end{align*}
gives, whenever $\eta_k(L+a_k)\le1/3$,
\begin{align}
 \E\|z_{k+1}-r_k\|^2
 \le\left(1-\frac43\eta_ka_k\right)\E\|z_k-r_k\|^2
 +\frac52\eta_k^2\sigma^2.
 \label{rr:eq:fbf-tracking}
\end{align}

\subsubsection{Effect of recentering}

For each $k$, choose $h_k\in A(r_k)$ such that
\begin{align}
 F(r_k)+h_k+a_k(r_k-c_k)=0.
 \label{rr:eq:root-selection}
\end{align}
The recentering update gives
\begin{align}
 a_{k+1}(z-c_{k+1})
 =a_k(z-c_k)+\lambda_k(z-z_{k+1}).
 \label{rr:eq:recenter-identity}
\end{align}
Subtracting \eqref{rr:eq:root-selection} at $k$ and $k+1$ and using \eqref{rr:eq:recenter-identity},
\begin{align}
 \lambda_k(z_{k+1}-r_k)
 =&F(r_{k+1})+h_{k+1}-F(r_k)-h_k\notag\\
 &+a_{k+1}(r_{k+1}-r_k).
 \label{rr:eq:root-difference}
\end{align}

\begin{lemma}[Bounds from recentering]
\label{rr:lem:transport}
For every RRSEG update,
\begin{align}
 \|z_{k+1}-r_{k+1}\|
 &\le\|z_{k+1}-r_k\|,
 \label{rr:eq:distance-transport}\\
 \|r_{k+1}-r_k\|
 &\le\frac{\lambda_k}{a_{k+1}}\|z_{k+1}-r_k\|.
 \label{rr:eq:root-motion}
\end{align}
Moreover, if
\begin{align*}
 g_k:=F(r_k)+h_k+a_0(r_k-z_0),
\end{align*}
then $g_0=0$ and
\begin{align}
 \|g_{k+1}-g_k\|
 \le\lambda_k\|z_{k+1}-r_k\|.
 \label{rr:eq:residual-transport}
\end{align}
\end{lemma}

\begin{proof}
The monotonicity of $F+A$ gives
\begin{align*}
 \left\langle F(r_{k+1})+h_{k+1}-F(r_k)-h_k,r_{k+1}-r_k\right\rangle\ge0.
\end{align*}
If $\lambda_k=0$, then \eqref{rr:eq:root-difference} and strong monotonicity give $r_{k+1}=r_k$, so all conclusions are immediate.
Suppose $\lambda_k>0$.
Taking the inner product of \eqref{rr:eq:root-difference} with $r_{k+1}-r_k$ gives
\begin{align*}
 a_{k+1}\|r_{k+1}-r_k\|^2
 \le\lambda_k\langle z_{k+1}-r_k,r_{k+1}-r_k\rangle,
\end{align*}
which proves \eqref{rr:eq:root-motion}.
Using the same inequality,
\begin{align*}
 &\|z_{k+1}-r_k\|^2-\|z_{k+1}-r_{k+1}\|^2\\
 &\quad=2\langle z_{k+1}-r_k,r_{k+1}-r_k\rangle-\|r_{k+1}-r_k\|^2\\
 &\quad\ge\left(\frac{2a_{k+1}}{\lambda_k}-1\right)\|r_{k+1}-r_k\|^2\ge0,
\end{align*}
which proves \eqref{rr:eq:distance-transport}.

The identity \eqref{rr:eq:root-difference} can also be written as
\begin{align*}
 \lambda_k(z_{k+1}-r_k)
 =(g_{k+1}-g_k)+(a_{k+1}-a_0)(r_{k+1}-r_k).
\end{align*}
Moreover,
\begin{align*}
 \langle g_{k+1}-g_k,r_{k+1}-r_k\rangle
 \ge a_0\|r_{k+1}-r_k\|^2\ge0.
\end{align*}
Therefore
\begin{align*}
 \lambda_k^2\|z_{k+1}-r_k\|^2
 \ge\|g_{k+1}-g_k\|^2,
\end{align*}
which proves \eqref{rr:eq:residual-transport}.
\end{proof}

\subsubsection{Schedule-free residual bounds}

\begin{proposition}[Residual bound at the regularized zero]
\label{rr:prop:schedule-free}
For every $m\ge1$,
\begin{align}
 \left(\E[\mathcal R(r_m)^2]\right)^{1/2}
 &\le\left(\E\|F(r_m)+h_m\|^2\right)^{1/2}\notag\\
 &\le a_0D+
 \sum_{k=0}^{m-1}\lambda_k
 \left(\E\|z_{k+1}-r_k\|^2\right)^{1/2}.
 \label{rr:eq:root-residual}
\end{align}
\end{proposition}

\begin{proof}
Let $z_\star$ solve the original inclusion and choose $h_\star\in A(z_\star)$ such that $F(z_\star)+h_\star=0$.
At $k=0$, \eqref{rr:eq:root-selection} gives
\begin{align*}
 F(r_0)+h_0=a_0(z_0-r_0).
\end{align*}
Monotonicity of $F+A$ gives
\begin{align*}
 0
 &\le\langle F(r_0)+h_0-F(z_\star)-h_\star,r_0-z_\star\rangle\\
 &=a_0\langle z_0-r_0,z_0-z_\star\rangle-a_0\|z_0-r_0\|^2,
\end{align*}
so
\begin{align}
 \|z_0-r_0\|\le D.
 \label{rr:eq:initial-root-distance}
\end{align}

By \eqref{rr:eq:residual-transport} and $g_0=0$,
\begin{align}
 \|g_m\|
 \le\sum_{k=0}^{m-1}\lambda_k\|z_{k+1}-r_k\|.
 \label{rr:eq:g-telescope}
\end{align}
Monotonicity of $F+A$ also gives
\begin{align*}
 \langle g_m,r_m-r_0\rangle
 \ge a_0\|r_m-r_0\|^2.
\end{align*}
Hence
\begin{align*}
 \|g_m-a_0(r_m-r_0)\|^2
 &\le\|g_m\|^2.
\end{align*}
Since
\begin{align*}
 F(r_m)+h_m
 =g_m-a_0(r_m-r_0)+a_0(z_0-r_0),
\end{align*}
we obtain
\begin{align*}
 \|F(r_m)+h_m\|
 \le a_0D+\|g_m\|.
\end{align*}
Combining this inequality with \eqref{rr:eq:g-telescope}, taking $L_2$ norms, and applying Minkowski's inequality proves \eqref{rr:eq:root-residual}.
\end{proof}

The final composite output will be obtained from the forward-backward displacement.

\begin{lemma}[Forward-backward displacement]
\label{rr:lem:fb-displacement}
Let
\begin{align*}
 p_m:=J_{\frac1L A}\left(z_m-\frac1L F(z_m)\right).
\end{align*}
Then
\begin{align}
 L\left(\E\|z_m-p_m\|^2\right)^{1/2}
 \le&\sqrt2L\left(\E\|z_m-r_m\|^2\right)^{1/2}
 +\left(\E\|F(r_m)+h_m\|^2\right)^{1/2}\notag\\
 \le&\sqrt2L\left(\E\|z_m-r_m\|^2\right)^{1/2}
 +a_0D
 +\sum_{k=0}^{m-1}\lambda_k
 \left(\E\|z_{k+1}-r_k\|^2\right)^{1/2}.
 \label{rr:eq:fb-residual}
\end{align}
\end{lemma}

\begin{proof}
Write
\begin{align*}
 d&:=z_m-r_m,
 &
 f&:=\frac{F(z_m)-F(r_m)}{L},
 &
 q&:=\frac{F(r_m)+h_m}{L}.
\end{align*}
Since $h_m\in A(r_m)$,
\begin{align*}
 r_m
 =J_{\frac1L A}\left(r_m-\frac1L F(r_m)+q\right).
\end{align*}
Let $s:=p_m-r_m$.
Firm nonexpansiveness of the resolvent gives
\begin{align*}
 \|s\|^2\le\langle s,d-f-q\rangle,
\end{align*}
which is equivalent to
\begin{align*}
 \left\|s-\frac12(d-f-q)\right\|
 \le\frac12\|d-f-q\|.
\end{align*}
Therefore
\begin{align*}
 \|z_m-p_m\|
 &=\|d-s\|\\
 &\le\frac12\bigl(\|d+f+q\|+\|d-f-q\|\bigr)\\
 &\le\sqrt{\|d\|^2+\|f+q\|^2}\\
 &\le\sqrt{\|d\|^2+\bigl(\|d\|+\|q\|\bigr)^2}\\
 &\le\sqrt2\|d\|+\|q\|.
\end{align*}
Here the third line follows from the parallelogram identity, the fourth uses the $L$-Lipschitz continuity of $F$, and the last uses
$\sqrt{a^2+(a+b)^2}\le\sqrt2a+b$ for $a,b\ge0$.
Taking $L_2$ norms gives the first inequality in \eqref{rr:eq:fb-residual}; the second follows from Proposition~\ref{rr:prop:schedule-free}.
\end{proof}

\subsection{Tracking bound and schedule}

The schedule-free bounds contain the sum
$\sum_k\lambda_k(\E\|z_{k+1}-r_k\|^2)^{1/2}$.
For the deterministic part, consider the envelope $(a_0/a_k)^pD^2$.
When $a_N=L$,
\begin{align*}
 a_0^{p/2}\sum_{k=0}^{N-1}\frac{\lambda_k}{a_{k+1}^{p/2}}
 &\le a_0^{p/2}\int_{a_0}^{L}\frac{da}{a^{p/2}}\\
 &=\frac{2L}{p-2}\left[
 \frac{a_0}{L}-\left(\frac{a_0}{L}\right)^{p/2}
 \right]
\end{align*}
for $p>2$, so any $p>2$ keeps this contribution of order $a_0D$.
For the stochastic part, an envelope proportional to $\sigma^2/a_k^2$ makes the same sum proportional to
$\sum_k\lambda_k/a_{k+1}$ and hence logarithmic.
This motivates the two envelopes used below.

Let $Q\ge1$, $\gamma>0$, and define
\begin{align}
 \eta_k&=\frac1{3(L+Qa_k)},\notag\\
 a_{k+1}
 &=\min\left\{L,\frac{a_k}{1-\gamma\eta_ka_k}\right\},
 \qquad
 \lambda_k=a_{k+1}-a_k.
 \label{rr:eq:power-schedule}
\end{align}
Then
\begin{align}
 \eta_k(L+a_k)&\le\frac13,
 &
 \eta_ka_k&\le\frac1{3(Q+1)}.
 \label{rr:eq:stepsize-bounds}
\end{align}

\begin{lemma}[One-step envelope bounds]
\label{rr:lem:power-closure}
Let $p>2$, $0<\gamma<2/3$, and $p\gamma\le4/3$.
Under \eqref{rr:eq:power-schedule},
\begin{align}
 \left(1-\frac43\eta_ka_k\right)
 \left(\frac{a_0}{a_k}\right)^p
 &\le\left(\frac{a_0}{a_{k+1}}\right)^p,
 \label{rr:eq:power-det}\\
 \left(1-\frac43\eta_ka_k\right)
 \frac{5}{4(2-3\gamma)(Q+1)a_k^2}
 +\frac52\eta_k^2
 &\le\frac{5}{4(2-3\gamma)(Q+1)a_{k+1}^2}.
 \label{rr:eq:power-noise}
\end{align}
\end{lemma}

\begin{proof}
The update gives
\begin{align*}
 \frac{a_k}{a_{k+1}}\ge1-\gamma\eta_ka_k.
\end{align*}
Bernoulli's inequality and $p\gamma\le4/3$ therefore give
\begin{align*}
 \left(\frac{a_k}{a_{k+1}}\right)^p
 \ge(1-\gamma\eta_ka_k)^p
 \ge1-p\gamma\eta_ka_k
 \ge1-\frac43\eta_ka_k,
\end{align*}
which proves \eqref{rr:eq:power-det}.

For \eqref{rr:eq:power-noise}, let $u_k:=1-a_k/a_{k+1}$.
Then $0\le u_k\le\gamma\eta_ka_k$ and
\begin{align*}
 \frac{a_k^2}{a_{k+1}^2}
 -\left(1-\frac43\eta_ka_k\right)
 &=\frac43\eta_ka_k-2u_k+u_k^2\\
 &\ge\left(\frac43-2\gamma\right)\eta_ka_k.
\end{align*}
Hence
\begin{align*}
 &\frac{5}{4(2-3\gamma)(Q+1)a_{k+1}^2}
 -\left(1-\frac43\eta_ka_k\right)
 \frac{5}{4(2-3\gamma)(Q+1)a_k^2}\\
 &\quad\ge\frac{5\eta_k}{6(Q+1)a_k}
 \ge\frac52\eta_k^2,
\end{align*}
where the last inequality follows from \eqref{rr:eq:stepsize-bounds}.
\end{proof}

\begin{lemma}[Tracking bound]
\label{rr:lem:power-tracking}
Let $p>2$, $0<\gamma<2/3$, and $p\gamma\le4/3$.
Under \eqref{rr:eq:power-schedule},
\begin{align}
 \E\|z_k-r_k\|^2
 &\le\left(\frac{a_0}{a_k}\right)^pD^2
 +\frac{5\sigma^2}{4(2-3\gamma)(Q+1)a_k^2},
 \label{rr:eq:power-tracking}\\
 \E\|z_{k+1}-r_k\|^2
 &\le\left(\frac{a_0}{a_{k+1}}\right)^pD^2
 +\frac{5\sigma^2}{4(2-3\gamma)(Q+1)a_{k+1}^2}.
 \label{rr:eq:power-pretransport}
\end{align}
\end{lemma}

\begin{proof}
At $k=0$, \eqref{rr:eq:initial-root-distance} gives $\|z_0-r_0\|\le D$.
Assume \eqref{rr:eq:power-tracking} holds at iteration $k$.
Using \eqref{rr:eq:fbf-tracking}, \eqref{rr:eq:power-det}, and \eqref{rr:eq:power-noise},
\begin{align*}
 \E\|z_{k+1}-r_k\|^2
 &\le\left(1-\frac43\eta_ka_k\right)\E\|z_k-r_k\|^2
 +\frac52\eta_k^2\sigma^2\\
 &\le\left(\frac{a_0}{a_{k+1}}\right)^pD^2
 +\frac{5\sigma^2}{4(2-3\gamma)(Q+1)a_{k+1}^2}.
\end{align*}
Thus \eqref{rr:eq:power-pretransport} holds.
Then \eqref{rr:eq:distance-transport} gives \eqref{rr:eq:power-tracking} at iteration $k+1$.
\end{proof}

\subsection{Reaching \texorpdfstring{$a_N=L$}{a\_N=L} and the core finite-horizon bound}

\begin{lemma}[Condition for $a_N=L$]
\label{rr:lem:traversal}
Suppose \eqref{rr:eq:power-schedule} is used for $N$ iterations and
\begin{align}
 \frac3\gamma\left(
 \frac{L}{a_0}-1+Q\log\frac{L}{a_0}
 \right)\le N.
 \label{rr:eq:traversal-condition}
\end{align}
Then $a_N=L$.
\end{lemma}

\begin{proof}
Assume $a_N<L$.
Then no clipping occurs and
\begin{align*}
 \frac1{a_{k+1}}=\frac1{a_k}-\gamma\eta_k,
 \qquad
 \frac{a_{k+1}}{a_k}=\frac1{1-\gamma\eta_ka_k}.
\end{align*}
Using $-\log(1-x)\ge x$,
\begin{align*}
 &\frac3\gamma\left[
 L\left(\frac1{a_k}-\frac1{a_{k+1}}\right)
 +Q\log\frac{a_{k+1}}{a_k}
 \right]\\
 &\quad=3L\eta_k
 +\frac{3Q}{\gamma}\bigl[-\log(1-\gamma\eta_ka_k)\bigr]\\
 &\quad\ge3\eta_k(L+Qa_k)=1.
\end{align*}
Summing over $k=0,\ldots,N-1$ gives
\begin{align*}
 \frac3\gamma\left[
 L\left(\frac1{a_0}-\frac1{a_N}\right)
 +Q\log\frac{a_N}{a_0}
 \right]\ge N.
\end{align*}
The expression in brackets is strictly increasing in $a_N$.
Since $a_N<L$, this contradicts \eqref{rr:eq:traversal-condition}.
\end{proof}

\begin{proposition}[Core finite-horizon bound]
\label{rr:prop:power-certificate}
Let $p>2$, $0<\gamma<2/3$, $p\gamma\le4/3$, $Q\ge1$, and $0<a_0<L$.
If \eqref{rr:eq:traversal-condition} holds, define
\begin{align*}
 p_N:=J_{\frac1L A}\left(z_N-\frac1L F(z_N)\right).
\end{align*}
Then
\begin{align}
 L\left(\E\|z_N-p_N\|^2\right)^{1/2}
 \le&LD\left[
 \frac{p}{p-2}\frac{a_0}{L}
 +\left(\sqrt2-\frac{2}{p-2}\right)
 \left(\frac{a_0}{L}\right)^{p/2}
 \right]\notag\\
 &+\sigma\sqrt{\frac{5}{4(2-3\gamma)(Q+1)}}
 \left(\log\frac{L}{a_0}+\sqrt2\right).
 \label{rr:eq:power-core}
\end{align}
\end{proposition}

\begin{proof}
Lemma~\ref{rr:lem:traversal} gives $a_N=L$.
Lemma~\ref{rr:lem:power-tracking} therefore gives
\begin{align*}
 L\left(\E\|z_N-r_N\|^2\right)^{1/2}
 \le LD\left(\frac{a_0}{L}\right)^{p/2}
 +\sigma\sqrt{\frac{5}{4(2-3\gamma)(Q+1)}}.
\end{align*}
Proposition~\ref{rr:prop:schedule-free} and \eqref{rr:eq:power-pretransport} give
\begin{align*}
 \left(\E\|F(r_N)+h_N\|^2\right)^{1/2}
 \le&a_0D
 +a_0^{p/2}D\sum_{k=0}^{N-1}\frac{\lambda_k}{a_{k+1}^{p/2}}\\
 &+\sigma\sqrt{\frac{5}{4(2-3\gamma)(Q+1)}}
 \sum_{k=0}^{N-1}\frac{\lambda_k}{a_{k+1}}.
\end{align*}
Since both integrands are decreasing,
\begin{align*}
 a_0^{p/2}\sum_{k=0}^{N-1}\frac{\lambda_k}{a_{k+1}^{p/2}}
 &\le\frac{2L}{p-2}\left[
 \frac{a_0}{L}-\left(\frac{a_0}{L}\right)^{p/2}
 \right],\\
 \sum_{k=0}^{N-1}\frac{\lambda_k}{a_{k+1}}
 &\le\log\frac{L}{a_0}.
\end{align*}
Thus
\begin{align*}
 \left(\E\|F(r_N)+h_N\|^2\right)^{1/2}
 \le&LD\left[
 \frac{p}{p-2}\frac{a_0}{L}
 -\frac{2}{p-2}\left(\frac{a_0}{L}\right)^{p/2}
 \right]\\
 &+\sigma\sqrt{\frac{5}{4(2-3\gamma)(Q+1)}}
 \log\frac{L}{a_0}.
\end{align*}
Substituting the last two bounds into Lemma~\ref{rr:lem:fb-displacement} proves \eqref{rr:eq:power-core}.
\end{proof}

\subsection{Choosing the quartic tracking envelope}

The admissible parameters satisfy $p>2$ and $p\gamma\le4/3$.
Solving \eqref{rr:eq:traversal-condition} for the largest admissible $Q$ gives
\begin{align*}
 Q
 =\frac{\frac{\gamma N}{3}-\left(\frac{L}{a_0}-1\right)}
 {\log\left(\frac{L}{a_0}\right)}.
\end{align*}
To compare the leading coefficients, write
\begin{align*}
 a_0=\frac{3\rho L}{\gamma N},
 \qquad \rho>1,
\end{align*}
and hold $\rho$ fixed.
Then
\begin{align*}
 Q
 =\frac{\gamma(\rho-1)}{3\rho}\frac{N}{\log N}\bigl(1+o(1)\bigr).
\end{align*}
The leading deterministic coefficient in Proposition~\ref{rr:prop:power-certificate} is
\begin{align*}
 \frac{3\rho}{N}\frac{p}{\gamma(p-2)}+o(N^{-1}).
\end{align*}
For fixed $\gamma$, $p/(p-2)$ decreases with $p$, so $p\gamma\le4/3$ selects
\begin{align*}
 p=\frac{4}{3\gamma}.
\end{align*}
With this choice, the leading deterministic coefficient and the square of the leading stochastic coefficient both depend on $\gamma$ through
\begin{align*}
 \frac1{\gamma(2-3\gamma)}.
\end{align*}
Indeed,
\begin{align*}
 \frac{p}{\gamma(p-2)}
 &=\frac{2}{\gamma(2-3\gamma)},
\end{align*}
while
\begin{align*}
 \frac{1}{(2-3\gamma)Q}
 =\frac{3\rho}{\gamma(2-3\gamma)(\rho-1)}
 \frac{\log N}{N}\bigl(1+o(1)\bigr).
\end{align*}
Since $\gamma(2-3\gamma)$ is maximized at $\gamma=1/3$, the corresponding value is $p=4$.
We therefore use the quartic deterministic envelope and allow every $0<\gamma\le1/3$ in the finite-horizon theorem; Corollary~\ref{rr:cor:explicit} uses $\gamma=1/3$.

For $p=4$, Proposition~\ref{rr:prop:power-certificate} becomes
\begin{align}
 L\left(\E\|z_N-p_N\|^2\right)^{1/2}
 \le&\left(2a_0+(\sqrt2-1)\frac{a_0^2}{L}\right)D\notag\\
 &+\sigma\sqrt{\frac{5}{4(2-3\gamma)(Q+1)}}
 \left(\log\frac{L}{a_0}+\sqrt2\right).
 \label{rr:eq:fb-final}
\end{align}

\subsection{Equation case}

The equation case uses the same one-step, recentering, tracking, and traversal bounds.
Only the final conversion is simpler.

\begin{corollary}[Equation-case last-iterate bound]
\label{rr:cor:equation}
Suppose $A=0$, $p=4$, $0<\gamma\le1/3$, and the conditions of Lemma~\ref{rr:lem:traversal} hold.
Then
\begin{align}
 \left(\E\|F(z_N)\|^2\right)^{1/2}
 \le2a_0D
 +\sigma\sqrt{\frac{5}{4(2-3\gamma)(Q+1)}}
 \left(\log\frac{L}{a_0}+1\right).
 \label{rr:eq:equation-general}
\end{align}
\end{corollary}

\begin{proof}
When $A=0$, $h_N=0$.
Proposition~\ref{rr:prop:schedule-free}, \eqref{rr:eq:power-pretransport}, and the two integral bounds used in the proof of Proposition~\ref{rr:prop:power-certificate} give
\begin{align*}
 \left(\E\|F(r_N)\|^2\right)^{1/2}
 \le\left(2a_0-\frac{a_0^2}{L}\right)D
 +\sigma\sqrt{\frac{5}{4(2-3\gamma)(Q+1)}}
 \log\frac{L}{a_0}.
\end{align*}
Lipschitz continuity and \eqref{rr:eq:power-tracking} give
\begin{align*}
 \left(\E\|F(z_N)-F(r_N)\|^2\right)^{1/2}
 \le\frac{a_0^2}{L}D
 +\sigma\sqrt{\frac{5}{4(2-3\gamma)(Q+1)}}.
\end{align*}
Adding these two inequalities proves \eqref{rr:eq:equation-general}.
\end{proof}

\subsection{Composite output}

Let
\begin{align*}
 m:=\lfloor\min\{Q,N\}\rfloor,
 \qquad
 \overline F_N:=\frac1m\sum_{k=N-m}^{N-1}F_{\xi_k}(z_k),
\end{align*}
and define
\begin{align*}
 \widehat z_N
 :=J_{\frac1L A}\left(z_N-\frac1L\overline F_N\right).
\end{align*}
The average uses oracle evaluations already computed by RRSEG.

\begin{lemma}[Accuracy of the averaged oracle evaluations]
\label{rr:lem:output-average}
Assume $p=4$, $0<\gamma\le1/3$, and $a_N=L$.
Then
\begin{align}
 \left(\E\|\overline F_N-F(z_N)\|^2\right)^{1/2}
 \le&\frac{77}{32}\frac{a_0^2}{L}D
 +\frac{\sigma}{\sqrt m}\notag\\
 &+\frac94\sigma
 \sqrt{\frac{5}{4(2-3\gamma)(Q+1)}}.
 \label{rr:eq:output-average}
\end{align}
\end{lemma}

\begin{proof}
From the update of $a_k$ and \eqref{rr:eq:stepsize-bounds},
\begin{align*}
 1-\frac{a_k}{a_{k+1}}
 \le\gamma\eta_ka_k
 \le\frac1{9(Q+1)}.
\end{align*}
Since $m\le Q$ and $a_N=L$,
\begin{align*}
 \frac{a_{N-m}}{L}
 &=\prod_{k=N-m}^{N-1}\frac{a_k}{a_{k+1}}\\
 &\ge1-\sum_{k=N-m}^{N-1}
 \left(1-\frac{a_k}{a_{k+1}}\right)\\
 &\ge1-\frac{m}{9(Q+1)}
 >\frac89.
\end{align*}
Thus, for $N-m\le k\le N$,
\begin{align}
 \left(\E\|z_k-r_k\|^2\right)^{1/2}
 \le\frac{81}{64}\frac{a_0^2}{L^2}D
 +\frac98\frac{\sigma}{L}
 \sqrt{\frac{5}{4(2-3\gamma)(Q+1)}}.
 \label{rr:eq:suffix-tracking}
\end{align}
The same bound holds for
$(\E\|z_{k+1}-r_k\|^2)^{1/2}$ when $N-m\le k\le N-1$.
At the terminal iterate, the sharper bound
\begin{align}
 \left(\E\|z_N-r_N\|^2\right)^{1/2}
 \le\frac{a_0^2}{L^2}D
 +\frac{\sigma}{L}
 \sqrt{\frac{5}{4(2-3\gamma)(Q+1)}}
 \label{rr:eq:terminal-tracking}
\end{align}
follows directly from \eqref{rr:eq:power-tracking} and $a_N=L$.

By \eqref{rr:eq:root-motion}, for every $N-m\le k\le N$,
\begin{align*}
 \left(\E\|r_k-r_N\|^2\right)^{1/2}
 &\le\sum_{j=k}^{N-1}\frac{\lambda_j}{a_{j+1}}
 \left(\E\|z_{j+1}-r_j\|^2\right)^{1/2}\\
 &<\frac19\left[
 \frac{81}{64}\frac{a_0^2}{L^2}D
 +\frac98\frac{\sigma}{L}
 \sqrt{\frac{5}{4(2-3\gamma)(Q+1)}}
 \right].
\end{align*}
Combining this inequality with \eqref{rr:eq:suffix-tracking} and \eqref{rr:eq:terminal-tracking},
\begin{align}
 \left(\E\|z_k-z_N\|^2\right)^{1/2}
 \le\frac{77}{32}\frac{a_0^2}{L^2}D
 +\frac94\frac{\sigma}{L}
 \sqrt{\frac{5}{4(2-3\gamma)(Q+1)}}.
 \label{rr:eq:suffix-clustering}
\end{align}
Indeed, the deterministic coefficient is
$(1+1/9)(81/64)+1=77/32$, and the stochastic coefficient is
$(1+1/9)(9/8)+1=9/4$.

Now decompose
\begin{align*}
 \overline F_N-F(z_N)
 =&\frac1m\sum_{k=N-m}^{N-1}\bigl(F_{\xi_k}(z_k)-F(z_k)\bigr)\\
 &+\frac1m\sum_{k=N-m}^{N-1}\bigl(F(z_k)-F(z_N)\bigr).
\end{align*}
Because $z_k$ is determined before $\xi_k$ is drawn, the errors in the first sum form a martingale difference sequence, so
\begin{align*}
 \E\left\|\frac1m\sum_{k=N-m}^{N-1}\bigl(F_{\xi_k}(z_k)-F(z_k)\bigr)\right\|^2
 \le\frac{\sigma^2}{m}.
\end{align*}
For the second sum, Lipschitz continuity and \eqref{rr:eq:suffix-clustering} give
\begin{align*}
 &\left(\E\left\|\frac1m\sum_{k=N-m}^{N-1}\bigl(F(z_k)-F(z_N)\bigr)\right\|^2\right)^{1/2}\\
 &\quad\le\frac{77}{32}\frac{a_0^2}{L}D
 +\frac94\sigma
 \sqrt{\frac{5}{4(2-3\gamma)(Q+1)}}.
\end{align*}
Minkowski's inequality proves \eqref{rr:eq:output-average}.
\end{proof}

\begin{lemma}[Residual of the final resolvent point]
\label{rr:lem:output-residual}
Let
\begin{align*}
 p_N=J_{\frac1L A}\left(z_N-\frac1L F(z_N)\right).
\end{align*}
Then
\begin{align}
 \mathcal R(\widehat z_N)
 \le\sqrt2L\|z_N-p_N\|
 +\sqrt2\|\overline F_N-F(z_N)\|.
 \label{rr:eq:output-residual}
\end{align}
\end{lemma}

\begin{proof}
The resolvent definitions give
\begin{align*}
 h_p&:=L(z_N-p_N)-F(z_N)\in A(p_N),\\
 h_{\widehat z}&:=L(z_N-\widehat z_N)-\overline F_N\in A(\widehat z_N).
\end{align*}
By monotonicity and Lipschitz continuity of $F$,
\begin{align*}
 \|F(p_N)+h_p\|^2
 &=\|L(z_N-p_N)-(F(z_N)-F(p_N))\|^2\\
 &\le2L^2\|z_N-p_N\|^2.
\end{align*}
Let $e:=\widehat z_N-p_N$ and $\delta:=\overline F_N-F(z_N)$.
Firm nonexpansiveness applied to the two resolvent inputs gives
\begin{align*}
 \|e\|^2\le-\frac1L\langle e,\delta\rangle.
\end{align*}
Therefore
\begin{align*}
 \left\|e+\frac\delta L\right\|^2
 \le\frac{\|\delta\|^2}{L^2}-\|e\|^2.
\end{align*}
Using Lipschitz continuity,
\begin{align*}
 &\|(F(\widehat z_N)+h_{\widehat z})-(F(p_N)+h_p)\|\\
 &\quad\le L\|e\|+L\left\|e+\frac\delta L\right\|\\
 &\quad\le\sqrt2\|\delta\|.
\end{align*}
Combining the last two bounds proves \eqref{rr:eq:output-residual}.
\end{proof}

\subsection{Final parameter choices}

Fix $0<\gamma\le1/3$ and $a_0\in(0,L)$.
For fixed $a_0$ and $\gamma$, increasing $Q$ decreases the stochastic term but slows the increase of $a_k$.
We therefore take the largest $Q$ allowed by \eqref{rr:eq:traversal-condition}:
\begin{align*}
 Q_{N,\gamma}(a_0)
 :=\frac{\frac{\gamma N}{3}-\left(\frac{L}{a_0}-1\right)}
 {\log\left(\frac{L}{a_0}\right)}.
\end{align*}
With $Q=Q_{N,\gamma}(a_0)$, Lemma~\ref{rr:lem:traversal} gives $a_N=L$.
Combining \eqref{rr:eq:fb-final}, Lemma~\ref{rr:lem:output-average}, and \eqref{rr:eq:output-residual} gives
\begin{align*}
 \left(\E[\mathcal R(\widehat z_N)^2]\right)^{1/2}
 \le&\sqrt2LD\left[
 2\frac{a_0}{L}
 +\left(\sqrt2-1+\frac{77}{32}\right)
 \left(\frac{a_0}{L}\right)^2
 \right]
 +\frac{\sqrt2\sigma}{\sqrt m}\\
 &+\sigma\sqrt{\frac{5}{2(2-3\gamma)(Q+1)}}
 \left(\log\frac{L}{a_0}+\sqrt2+\frac94\right),
\end{align*}
where $m=\lfloor\min\{Q,N\}\rfloor$.
This proves Theorem~\ref{rr:thm:main}.

For Corollary~\ref{rr:cor:explicit}, choose
\begin{align*}
 \gamma=\frac13,
 \qquad
 a_0=\frac{18L}{N}.
\end{align*}
Then
\begin{align*}
 Q_N
 =Q_{N,1/3}\left(\frac{18L}{N}\right)
 =\frac{\frac{N}{18}+1}{\log\left(\frac{N}{18}\right)}.
\end{align*}
For $N\ge19$, $N/18>1$ and
$0<\log(N/18)<N/18<N/18+1$, so $Q_N>1$.
Substituting these values into Theorem~\ref{rr:thm:main} gives
\begin{align*}
 \left(\E[\mathcal R(\widehat z_N)^2]\right)^{1/2}
 \le&\frac{36\sqrt2LD}{N}
 +\frac{324\sqrt2LD}{N^2}
 \left(\sqrt2-1+\frac{77}{32}\right)
 +\frac{\sqrt2\sigma}{\sqrt{m_N}}\\
 &+\sigma\sqrt{\frac{5}{2(Q_N+1)}}
 \left(\log\frac{N}{18}+\sqrt2+\frac94\right),
\end{align*}
where
\begin{align*}
 m_N=\lfloor\min\{Q_N,N\}\rfloor.
\end{align*}
Since $Q_N=\Theta(N/\log N)$ and $m_N=\Theta(N/\log N)$,
\begin{align*}
 \E[\mathcal R(\widehat z_N)^2]
 =O\left(\frac{L^2D^2}{N^2}+\frac{\sigma^2\log^3N}{N}\right).
\end{align*}

Finally, applying Corollary~\ref{rr:cor:equation} to the same concrete schedule gives the sharper equation-case last-iterate bound
\begin{align}
 \left(\E\|F(z_N)\|^2\right)^{1/2}
 \le\frac{36LD}{N}
 +\sigma\sqrt{\frac{5}{4(Q_N+1)}}
 \left(\log\frac{N}{18}+1\right).
 \label{rr:eq:equation-bound}
\end{align}

\subsection{Shape of the concrete schedule}\label{app:rrseg-schedule}

For the parameters in Corollary~\ref{rr:cor:explicit},
$a_k$ is nondecreasing and remains at most $L$.
Before it reaches $L$,
\begin{align*}
a_{k+1}-a_k
&= \frac{\frac13\eta_k a_k^2}
{1-\frac13\eta_k a_k}
= \frac{a_k^2}{9L+(9Q_N-1)a_k}.
\end{align*}
The right-hand side increases with $a_k$, so the regularization
parameter grows more rapidly as it approaches $L$.
Figure~\ref{fig:rrseg-schedule} illustrates this behavior for the
horizon used in the experiments.

\begin{figure}[H]
 \centering
 \includegraphics[width=0.8\linewidth]{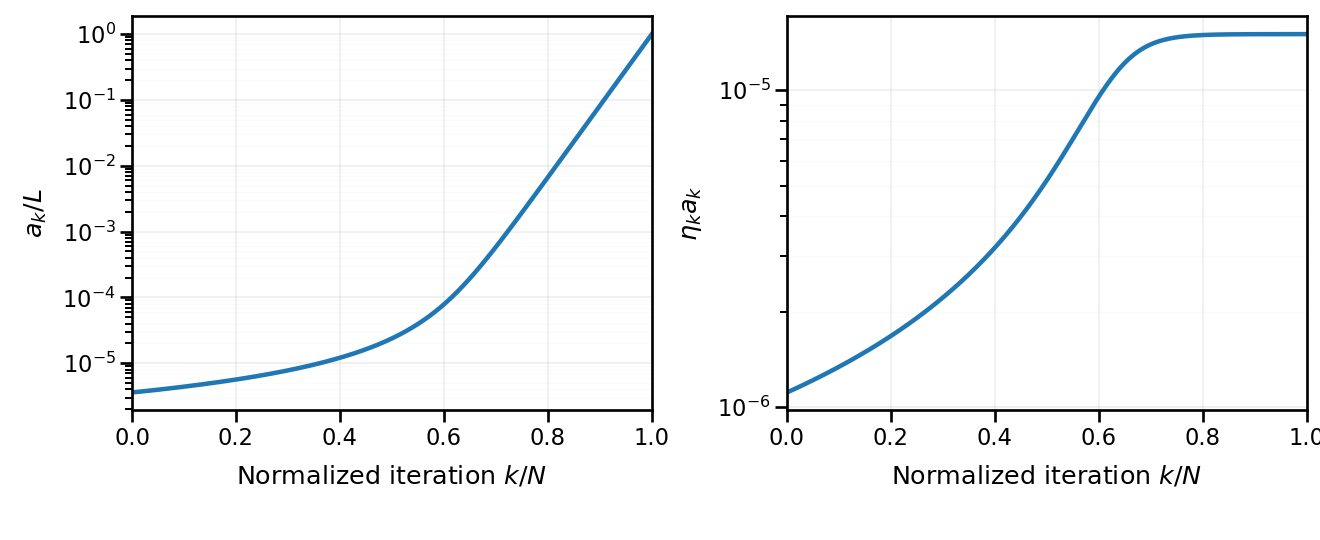}
 \caption{
 RRSEG parameters from Corollary~\ref{rr:cor:explicit} for
$N=5{,}000{,}000$, the horizon used for the solid RRSEG trajectories
in the experiments.
Left: $a_k/L$. Right: $\eta_k a_k$.
 }
 \label{fig:rrseg-schedule}
\end{figure}

%% file: appendices/experiments_full.tex
\subsection{Evaluation protocol and parameter choices}\label{app:experiment-protocol}
Problems~1 and~2 and rock-paper-scissors use $A=0$, while Problem~3 uses a nonzero full-domain nonsmooth operator. Each $F_\xi(z)$ evaluation counts as one oracle query; every method uses $50$ runs and at most $10^7$ oracle queries per run.

RRSEG uses Corollary~\ref{rr:cor:explicit} with the single fixed horizon $N=5{,}000{,}000$ on all four problems. Every solid RRSEG curve records the intermediate iterates $z_k$ from that run.
The stars in Figure~\ref{fig:main-experiments} use separate RRSEG runs with horizon $N=Q/2$ at $Q\in\{10^4, 10^5, 10^6, 10^7\}$ oracle queries. Problems~1 and~2 and rock-paper-scissors use the terminal iterate $z_N$, while Problem~3 uses the corrected output $\widehat{z}_N$.

On Problems~1 and~2 and rock-paper-scissors, BC-SEG+ uses $\gamma=1/(2L)$ and $\alpha_k=1/[18(k/100+1)]$.
On Problem~3, BC-PSEG+ uses $\gamma=1/(2L)$ and $\alpha_k=1/[18\sqrt{k/100+1}]$.
Stochastic GOMA is run as in Theorem~4 of \citet{sohrabi2026accelerated}.
Stochastic GOMA and RAIN-SL apply to equations, while the composite guarantee of E-Halpern is restricted to constrained problems; none is run on Problem~3.

For RAIN-SL, we use Algorithm~9 of \citet{chen2024nearoptimal}, the single-loop experimental variant, and tune $(\eta,\lambda,\gamma)$ over
\begin{align*}
 \eta&\in\{0.005,0.01,0.05,0.1,1,5,10\},\qquad
 \lambda\in\{0.001,0.01,0.1,1\},\qquad
 \gamma\in\{0.001,0.01,0.1,1\}.
\end{align*}
All candidates use the same tuning seed within each problem, disjoint from the evaluation seeds.
No candidate remains finite after $10^6$ or $10^5$ oracle queries, so we select the finite triple with the smallest squared residual after $12{,}000$ oracle queries.
The selected triple is $(0.05,0.001,0.001)$ for Problem~1 and rock-paper-scissors and $(0.1,0.001,0.001)$ for Problem~2.

For E-Halpern, we use Algorithm~2 of \citet{cai2022stochastic} and select $\epsilon$ using $50$ independent calibration runs.
Among $100$ logarithmically spaced values from $0.8\mathcal R(z_0)$ to $\max\{10^{-4}\mathcal R(z_0),10^{-6}\}$, we choose the smallest $\epsilon$ whose prescribed schedule uses at most $10^7$ oracle queries on average.
The selected values are $0.00413019$, $0.00922484$, and $0.02694452$ for Problems~1, 2, and rock-paper-scissors, respectively.
Each evaluation run stops before a batch would make the total exceed $10^7$ oracle queries.

\subsection{Plot construction}\label{app:experiment-plots}
For E-Halpern, the query count of a completed iterate can differ across runs; at each checkpoint we use the last completed iterate within that budget and plot it at the mean actual query count.

All fixed-query solid curves use the same logarithmic checkpoints, include the common initial point, and are obtained from one trajectory.

The dashed curves use Theorem~\ref{thm:inclusion} for VRAF and Theorem~4 of \citet{sohrabi2026accelerated} for stochastic GOMA.
Following the experiments of \citet{pethick2023solving}, the BC-SEG+ solid curves show $z_k$, while the BC-PSEG+ solid curve shows the resolvent point $\bar z_k$. Their dashed curves use Theorems~6.1 and~7.1, respectively, which bound the corresponding $\alpha_k$-weighted randomized outputs. For BC-PSEG+ on Problem~3, with $\alpha_0=1/18$, $\gamma=1/2$, $L_{\widehat F}=\sqrt{1.01}$, and $\rho=0$, the condition of Theorem~7.1 holds with $\eta\simeq8.8995$ and $\mu\simeq0.3133$.
For RRSEG, Problems~1 and~2 use the sharper $A=0$ last-iterate bound \eqref{rr:eq:equation-bound}, while Problem~3 uses Corollary~\ref{rr:cor:explicit}.
For Problems~1 and~2, at each horizon $N\ge19$ the sharper bound is
\begin{align*}
 \left(\E\|F(z_N)\|^2\right)^{1/2}
 \le\frac{36LD}{N}
 +\sigma\sqrt{\frac{5}{4(Q_N+1)}}
 \left(\log\frac{N}{18}+1\right),
 \qquad
 Q_N=\frac{N/18+1}{\log(N/18)}.
\end{align*}
The bounds are instantiated separately at each horizon and squared for the squared-residual axis.
These dashed RRSEG points are horizon-by-horizon theoretical guarantees; every empirical RRSEG solid curve comes from the single run described above.
The displayed BC-SEG+ and stochastic GOMA assumptions hold on Problems~1 and~2 with $L_{\widehat F}=1.2$ and $\kappa=1$, respectively.

\subsection{VRAF stepsize sensitivity}\label{app:vraf-stepsize-sensitivity}
To assess VRAF (practical), we multiply the theorem-prescribed $\alpha_k$ by a common $c\in\{1,2,4,8\}$ and leave all other coefficients unchanged. The same $c$ is used on every problem, and the four curves share the same $50$ seeds within each problem. Thus this is a sensitivity diagnostic rather than per-problem tuning; $c=1$ is VRAF and $c=4$ is VRAF (practical).
Since the VRAF schedule is initialized by $\alpha_0=7/(12\bar L)$ and the subsequent recurrence is multiplicative, multiplying all $\alpha_k$ by $c$ is equivalent to using $\bar L/c$ in the stepsize schedule. The experiment therefore also measures robustness to multiplicative misspecification of the Lipschitz scale used by VRAF.

\begin{figure}[H]
 \centering
 \includegraphics[width=.9\linewidth]{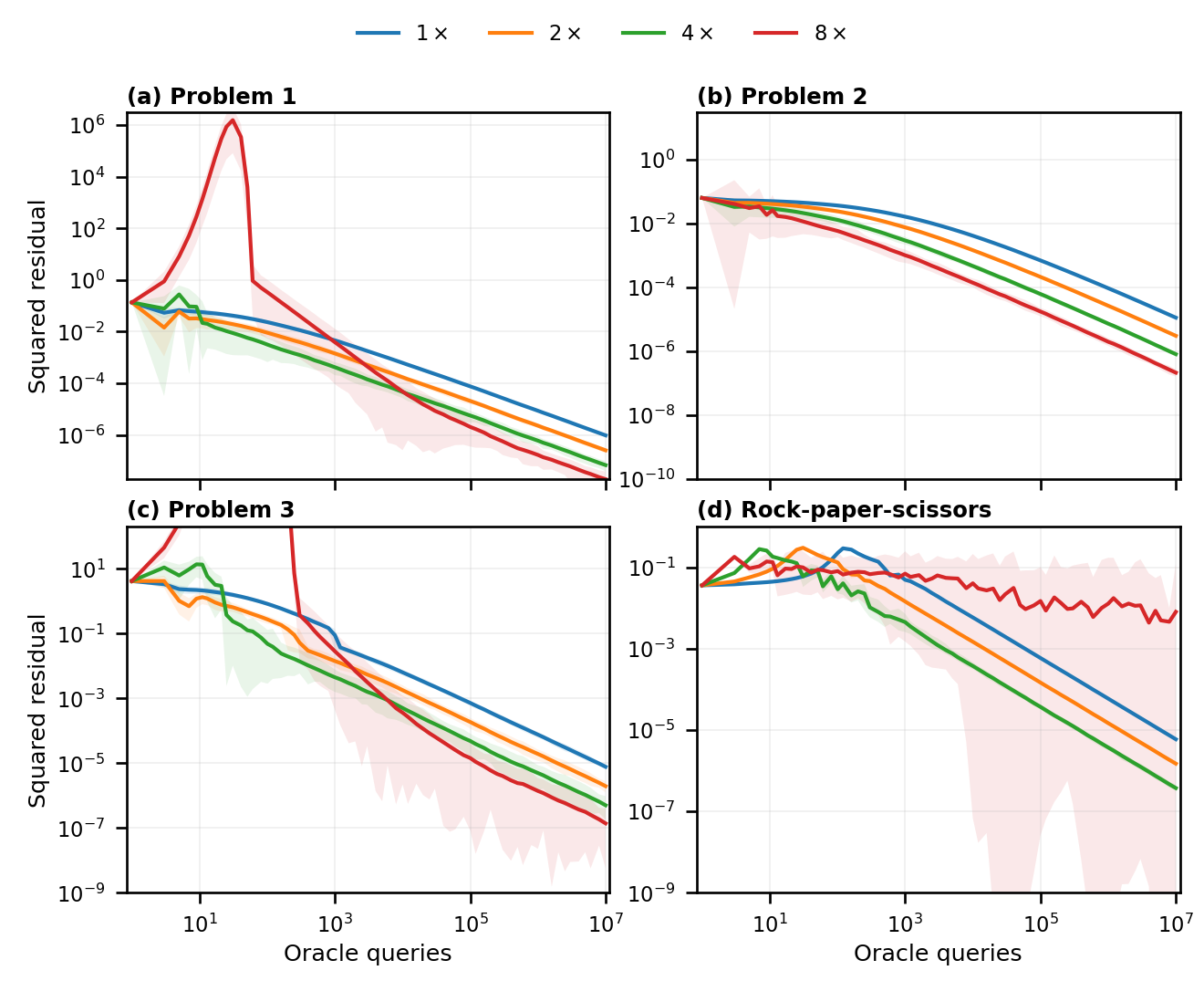}
 \caption{VRAF sensitivity to a common multiplier on $\alpha_k$. Curves show the mean over $50$ runs with pointwise $2.5\%$--$97.5\%$ quantiles.}
 \label{fig:vraf-stepsize-sensitivity}
\end{figure}

\subsection{Exact test problems and stochastic oracles}\label{app:experiment-problems}
For Problems~1 and~2, write $z=(z^{(1)},z^{(2)})$ and let $\phi(t)=t-\arctan t$. Both use $A=0$, are monotone and $1$-Lipschitz, and have $L_\Delta=0.2$ and $\bar L=\sqrt{1.04}$.

\emph{Problem 1.}
We use $z_\star=(0.4,0.6)$, $z_0=(0.95,0.05)$, and
\begin{align*}
 F(z)=\frac23\left(
 \phi\bigl(z^{(1)}-0.4\bigr)+\frac14\phi\bigl(z^{(1)}+z^{(2)}-1\bigr),\;
 z^{(2)}-0.6+\frac14\phi\bigl(z^{(1)}+z^{(2)}-1\bigr)
 \right).
\end{align*}
Here $\nabla F(z_\star)=\operatorname{diag}(0,2/3)$, so the first coordinate is locally flat.
The stochastic oracle is
\begin{align*}
 F_\xi(z)=F(z)
 &+0.1\bigl(\rho_1\tanh(z^{(1)}-0.4),\rho_2\tanh(z^{(2)}-0.6)\bigr)\\
 &+0.05\rho_3\tanh\bigl(z^{(1)}+z^{(2)}-1\bigr)(1,1)
 +0.08(\rho_4,\rho_5),
\end{align*}
where $\rho_1,\ldots,\rho_5$ are independent Rademacher variables; we use $\sigma^2=0.0378$.

\emph{Problem 2.}
We use
\begin{align*}
 F(z)&=\frac15\left(
 \phi(z^{(1)})+2\phi(z^{(1)}+z^{(2)}),\;
 \phi(z^{(2)})+2\phi(z^{(1)}+z^{(2)})
 \right),\\
 z_\star&=(0,0),
 & z_0&=(2,-2).
\end{align*}
Here $\nabla F(z_\star)=0$.
The stochastic oracle is
\begin{align*}
 F_\xi(z)=F(z)
 &+0.1\bigl(\rho_1\tanh z^{(1)},\rho_2\tanh z^{(2)}\bigr)\\
 &+0.05\rho_3\tanh(z^{(1)}+z^{(2)})(1,1)
 +0.08(\rho_4,\rho_5),
\end{align*}
where $\rho_1,\ldots,\rho_5$ are independent Rademacher variables; we use $\sigma^2=0.0378$.

\emph{Problem 3.}
Let $g(z)=\frac13|z^{(1)}|$, $A=\partial g$, and $\phi(t)=t-\arctan t$. We use
\begin{align*}
 F(z)
 &=\frac13\left(
 1+\phi(z^{(1)})+2z^{(2)},\;
 z^{(2)}-2z^{(1)}
 \right),\\
 z_\star&=(0,0),
 &z_0&=(2,-2).
\end{align*}
Since $F(z_\star)=(1/3,0)$ and $A(z_\star)=[-1/3,1/3]\times\{0\}$, $z_\star$ solves the inclusion. Moreover,
\begin{align*}
 \frac{\nabla F(z)+\nabla F(z)^\top}{2}
 &=\frac13\begin{pmatrix}
 \phi'(z^{(1)})&0\\
 0&1
 \end{pmatrix}\succeq0,
\end{align*}
Since $0\le\phi'(t)\le1$, $F$ is monotone and $\|\nabla F(z)\|\le(1+2)/3=1$. The inclusion is not locally strongly monotone: for $t>0$, the common selection $(1/3,0)\in A(t,0)\cap A(0,0)$ gives the monotonicity inner product $t\phi(t)/3=o(t^2)$. The resolvent is
\begin{align*}
 J_{\alpha A}(x,y)
 =\left(\operatorname{sign}(x)\max\{|x|-\alpha/3,0\},y\right),
\end{align*}
so $A$ has full domain and the RRSEG intermediate iterates have finite residuals. The stochastic oracle is
\begin{align*}
 F_\xi(z)=F(z)
 +0.1\bigl(\rho_1\tanh z^{(1)},\rho_2\tanh z^{(2)}\bigr)
 +0.35(\rho_3,\rho_4).
\end{align*}
Here $\rho_1,\ldots,\rho_4$ are independent Rademacher variables, $L_\Delta=0.1$, $\bar L=\sqrt{1.01}$, and $\sigma^2=0.265$. The residual is
\begin{align*}
 \mathcal R(z)^2
 =\operatorname{dist}\!\left(-F_1(z),\tfrac13\partial|z^{(1)}|\right)^2+F_2(z)^2.
\end{align*}

\emph{Rock-paper-scissors.}
Let
\begin{align*}
 M_{\rm RPS}
 &=
 \begin{pmatrix}
  0&-1&1\\
  1&0&-1\\
  -1&1&0
 \end{pmatrix},
 & p&=\operatorname{softmax}(x),
 & q&=\operatorname{softmax}(y),
\end{align*}
and define $F(x,y)=(\nabla_x\Psi(x,y),-\nabla_y\Psi(x,y))$ for $\Psi(x,y)=p^\top M_{\rm RPS}q$.
We use $x_\star=y_\star=0$, $x_0=(2,-1,-1)$, and $y_0=(-1,2,-1)$.
The stochastic oracle replaces $M_{\rm RPS}$ by
\begin{align*}
 M_\xi=M_{\rm RPS}+0.3\rho\,\operatorname{diag}(1,-1,0),
\end{align*}
where $\rho$ is Rademacher, and otherwise uses the same operator.
We use $L=0.5$, $L_\Delta=0.075$, and $\sigma^2=0.045$; the logit-space operator is nonmonotone.

\subsection{Rock-paper-scissors gap diagnostic}\label{app:rps-gap}
Softmax saturation can make the logit-space residual small even when the mixed strategies remain exploitable.
We therefore also report
\begin{align*}
 \operatorname{Gap}(p,q)
 &=\max_{q'\in\Delta_3}p^\top M_{\rm RPS}q'
 -\min_{p'\in\Delta_3}{p'}^\top M_{\rm RPS}q\\
 &=-\min_i(M_{\rm RPS}p)_i-\min_i(M_{\rm RPS}q)_i,
\end{align*}
which vanishes at the uniform Nash equilibrium.

\begin{figure}[H]
 \centering
 \includegraphics[width=0.62\linewidth]{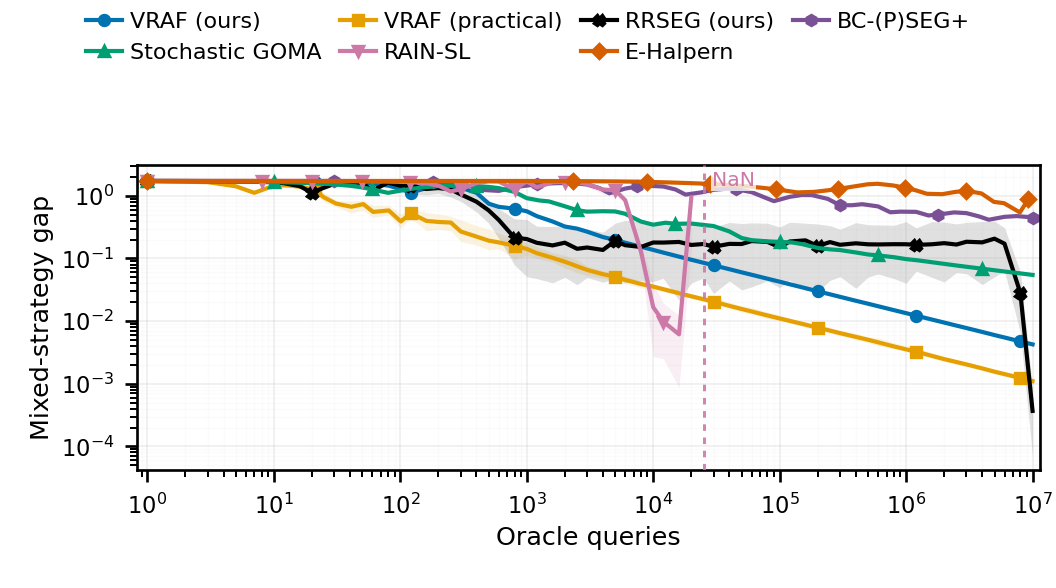}
 \caption{Mixed-strategy saddle-point gap for rock-paper-scissors versus the number of oracle queries. Solid curves and shaded regions use the same $50$ runs and quantile convention as Figure~\ref{fig:main-experiments}.}
 \label{fig:rps-gap}
\end{figure}

%% file: appendices/discussion_full.tex
\subsection{Anytime wrapper for RRSEG}\label{app:rrseg-anytime}
A standard doubling wrapper removes the need to specify a final horizon in advance.
Fix $N_0\ge 19$ and, for $s=0,1,\ldots$, run Corollary~\ref{rr:cor:explicit}
from $z_0$ with horizon
\[
N_s=2^sN_0.
\]
The runs are executed sequentially, and while the run with horizon $N_{s+1}$
is in progress, the method returns the terminal output of the completed run
with horizon $N_s$.

Suppose that a total of $T$ RRSEG iterations have been performed and that
the run with horizon $N_s$ is the most recently completed one. Then
\[
\sum_{j=0}^s N_j
\le T
<
\sum_{j=0}^{s+1}N_j
<
4N_s.
\]
Hence $N_s>T/4$. Applying Corollary~\ref{rr:cor:explicit} to the returned output gives
\[
\E[\mathcal R(\widehat z_{N_s})^2]
=
O\!\left(
\frac{L^2D^2}{T^2}
+
\frac{\sigma^2\log^3 T}{T}
\right).
\]
Each RRSEG iteration still uses two stochastic oracle queries, so the wrapper
remains fixed-query.

\subsection{VRAF need not converge pointwise}\label{app:discussion-vraf-pointwise}
The residual guarantee of VRAF does not imply convergence of its iterates.
The failure already occurs for a state-independent Gaussian oracle when every sampled operator is monotone.

\begin{proposition}[Pointwise failure of VRAF]\label{prop:vraf-pointwise-failure}
Let $\cH=\mathbb R$, $A=0$, $F\equiv0$, and $z_0=0$.
Use the valid Lipschitz upper bound $L=1$ and the state-independent Gaussian oracle
\begin{align*}
 F_\xi(z) = \xi,
 \qquad \xi\sim\mathcal N(0,\sigma^2),
 \qquad \sigma>0.
\end{align*}
Then Assumptions~\ref{ass:problem}, \ref{ass:same-sample}, and \ref{ass:noise-lipschitz} hold with $L_\Delta=0$,
and every sampled operator $F_\xi$ is monotone.
VRAF with the schedule of Theorem~\ref{thm:inclusion} satisfies
\begin{align*}
 \mathcal R(z_k)=0
 \qquad \text{for every }k,
\end{align*}
but $(z_k)$ does not converge in probability.
In particular, it does not converge strongly almost surely.
\end{proposition}

\begin{proof}
Write $X_k:=z_k$ and $W_k:=v_k$.
Since the same realization is used at $z_k$ and $z_{k+1}$,
the VRAF recursion becomes
\begin{align}
 X_{k+1} &= \frac{k}{k+3}X_k-\alpha_kW_k,
 &
 W_{k+1} &= (1-\gamma_{k+1})W_k+\gamma_{k+1}\xi_{k+1},
 \label{eq:vraf-pointwise-scalar}
\end{align}
with $W_0=\xi_0$.
Iterating the first recursion gives, for $n\ge1$,
\begin{align}
 X_n = - \frac{1}{n(n+1)(n+2)} \sum_{j=0}^{n-1}(j+1)(j+2)(j+3)\alpha_jW_j.
 \label{eq:vraf-pointwise-X-expansion}
\end{align}

For $m\ge1$, the schedule gives
\begin{align*}
 \gamma_m &= \frac{4m+9}{4(m+2)(m+3)},
 &
 1-\gamma_m &= \frac{(2m+3)(2m+5)}{4(m+2)(m+3)}.
\end{align*}
Hence
\begin{align}
 \gamma_m &\ge \frac{1}{2(m+3)},
 &
 1-\gamma_m &\ge \frac{m}{m+1}.
 \label{eq:vraf-pointwise-gamma-bounds}
\end{align}
For $1\le i\le j$, let $r_{j,i}$ be the coefficient of $\xi_i$ in $W_j$.
Equation~\eqref{eq:vraf-pointwise-scalar} and \eqref{eq:vraf-pointwise-gamma-bounds} give
\begin{align}
 r_{j,i}
 = \gamma_i\prod_{m=i+1}^{j}(1-\gamma_m)
 \ge \frac{i+1}{2(i+3)(j+1)}.
 \label{eq:vraf-pointwise-r-lower}
\end{align}

Let $q_{n,i}$ denote the coefficient of $\xi_i$ in $X_n$.  By \eqref{eq:vraf-pointwise-X-expansion},
\begin{align}
 |q_{n,i}|
 = \frac{1}{n(n+1)(n+2)} \sum_{j=i}^{n-1}(j+1)(j+2)(j+3)\alpha_jr_{j,i}.
 \label{eq:vraf-pointwise-q}
\end{align}
The Wallis estimate \eqref{vh:eq:scale-lower} and $\beta_j=3/(j+3)$ imply
\begin{align*}
 \alpha_j^2 \ge \frac{225\pi}{768(j+3)}.
\end{align*}
Set $c_\alpha:=\sqrt{75\pi/768}$.  For $k\ge6$, take
\begin{align*}
 k\le i\le\left\lfloor\frac{5k}{4}\right\rfloor,
 \qquad \left\lceil\frac{3k}{2}\right\rceil\le j\le2k-1.
\end{align*}
Then $j\ge i$, $\alpha_j\ge c_\alpha/\sqrt{k}$, and \eqref{eq:vraf-pointwise-r-lower} gives $r_{j,i}\ge1/(8k)$.
Moreover,
\begin{align*}
 \frac{(j+1)(j+2)(j+3)}{2k(2k+1)(2k+2)}\ge\frac18.
\end{align*}
There are at least $k/3$ admissible indices $j$.
Therefore \eqref{eq:vraf-pointwise-q} yields
\begin{align}
 |q_{2k,i}|\ge\frac{c_\alpha}{192\sqrt{k}}
 \qquad \left(k\le i\le\left\lfloor\frac{5k}{4}\right\rfloor\right).
 \label{eq:vraf-pointwise-q-lower}
\end{align}

Let $\mathcal F_k:=\sigma(\xi_0,\ldots,\xi_{k-1})$.
The random variable $X_k$ is $\mathcal F_k$-measurable, whereas the noises $\xi_i$ with $i\ge k$ are independent of $\mathcal F_k$.
Since all noises are Gaussian, $X_{2k}-X_k$ conditioned on $\mathcal F_k$ is Gaussian.
Equation~\eqref{eq:vraf-pointwise-q-lower} gives
\begin{align*}
 \operatorname{Var}(X_{2k}-X_k\mid\mathcal F_k)
 &\ge \sigma^2 \sum_{i=k}^{\lfloor5k/4\rfloor}q_{2k,i}^2\\
 &\ge \frac{c_\alpha^2}{4\cdot192^2}\sigma^2 =: \nu^2 > 0.
\end{align*}
For a Gaussian variable with variance at least $\nu^2$,
the probability of lying in an interval of length $\nu$ is maximized when its mean is the center of the interval.
Thus
\begin{align*}
 \mathbb P\left( |X_{2k}-X_k|\le\frac\nu2 \middle|\mathcal F_k \right)
 \le 2\Phi(1/2)-1<1,
\end{align*}
where $\Phi$ is the standard normal distribution function.
Hence $X_{2k}-X_k$ does not converge to zero in probability.
The sequence $(X_k)$ is not Cauchy in probability and therefore cannot converge in probability.

Finally, $F\equiv0$ and $A=0$, so every point is a solution and $\mathcal R(X_k)=0$ for every $k$.
\end{proof}

The obstruction is tangential stochastic motion along the solution set.
In the counterexample, the residual is identically zero, so no residual estimate can detect this motion.
Any pointwise-convergence result for VRAF therefore requires an additional mechanism or stronger assumptions.

%% file: references.bib
@inproceedings{alacaoglu2022stochastic,
  title = {Stochastic Variance Reduction for Variational Inequality Methods},
  author = {Alacaoglu, Ahmet and Malitsky, Yura},
  booktitle = {Conference on {{Learning Theory}}},
  pages = {778--816},
  year = 2022,
  publisher = {PMLR}
}

@article{alcala2023moving,
  title = {Moving Anchor Extragradient Methods for Smooth Structured Minimax Problems},
  author = {Alcala, James K and Chow, Yat Tin and Sunkula, Mahesh},
  journal = {arXiv preprint arXiv:2308.12359},
  year = 2023
}

@article{alcala2025stochastic,
  title = {Stochastic {{Moving Anchor Algorithms}} and a {{Popov}}'s {{Scheme}} with {{Moving Anchor}}},
  author = {Alcala, James and Chow, Yat Tin and Sunkula, Mahesh},
  journal = {arXiv preprint arXiv:2506.07290},
  year = 2025
}

@article{allen-zhu2018how,
  title = {How to Make the Gradients Small Stochastically: {{Even}} Faster Convex and Nonconvex Sgd},
  author = {{Allen-Zhu}, Zeyuan},
  journal = {Advances in Neural Information Processing Systems},
  volume = {31},
  year = 2018
}

@inproceedings{arjovsky2017wasserstein,
  title = {Wasserstein Generative Adversarial Networks},
  author = {Arjovsky, Martin and Chintala, Soumith and Bottou, L{\'e}on},
  booktitle = {International Conference on Machine Learning},
  pages = {214--223},
  year = 2017,
  publisher = {Pmlr}
}

@article{bassily2024private,
  title = {Private Algorithms for Stochastic Saddle Points and Variational Inequalities: {{Beyond}} Euclidean Geometry},
  author = {Bassily, Raef and Guzm{\'a}n, Crist{\'o}bal and Menart, Michael},
  journal = {Advances in Neural Information Processing Systems},
  volume = {37},
  pages = {128603--128635},
  year = 2024
}

@article{bot2026extragradient,
  title = {Extragradient {{Method}} with {{Flexible Anchoring}}: {{Strong Convergence}} and {{Fast Residual Decay}}},
  author = {Bo{\c t}, Radu I. and Chenchene, Enis},
  journal = {SIAM Journal on Optimization},
  volume = {36},
  number = {3},
  pages = {1420--1445},
  year = 2026
}

@article{cai2022stochastic,
  title = {Stochastic {{Halpern}} Iteration with Variance Reduction for Stochastic Monotone Inclusions},
  author = {Cai, Xufeng and Song, Chaobing and Guzm{\'a}n, Crist{\'o}bal and Diakonikolas, Jelena},
  journal = {Advances in Neural Information Processing Systems},
  volume = {35},
  pages = {24766--24779},
  year = 2022
}

@inproceedings{cai2023doubly,
  title = {Doubly Optimal No-Regret Learning in Monotone Games},
  author = {Cai, Yang and Zheng, Weiqiang},
  booktitle = {International {{Conference}} on {{Machine Learning}}},
  pages = {3507--3524},
  year = 2023,
  publisher = {PMLR}
}

@inproceedings{cai2024accelerated,
  title = {Accelerated {{Algorithms}} for {{Constrained Nonconvex-Nonconcave Min-Max Optimization}} and {{Comonotone Inclusion}}},
  author = {Cai, Yang and Oikonomou, Argyris and Zheng, Weiqiang},
  booktitle = {Proceedings of the 41st {{International Conference}} on {{Machine Learning}}},
  pages = {5312--5347},
  year = 2024,
  publisher = {PMLR}
}

@article{cai2026lastiterate,
  title = {Last-{{Iterate Convergence}} of {{Anchored Gradient Descent}}},
  author = {Cai, Yang and Zheng, Weiqiang},
  journal = {arXiv preprint arXiv:2604.12235},
  year = 2026
}

@article{chen2024nearoptimal,
  title = {Near-Optimal Algorithms for Making the Gradient Small in Stochastic Minimax Optimization},
  author = {Chen, Lesi and Luo, Luo},
  journal = {Journal of Machine Learning Research},
  volume = {25},
  number = {387},
  pages = {1--44},
  year = 2024
}

@article{chen2026unifying,
  title = {A {{Unifying View}} of {{Anchoring}} via {{Operator-Side Tikhonov Regularization}}},
  author = {Chen, Zihao},
  journal = {arXiv preprint arXiv:2605.30905},
  year = 2026
}

@article{cutkosky2019momentumbased,
  title = {Momentum-Based Variance Reduction in Non-Convex Sgd},
  author = {Cutkosky, Ashok and Orabona, Francesco},
  journal = {Advances in neural information processing systems},
  volume = {32},
  year = 2019
}

@inproceedings{diakonikolas2020halpern,
  title = {Halpern Iteration for Near-Optimal and Parameter-Free Monotone Inclusion and Strong Solutions to Variational Inequalities},
  author = {Diakonikolas, Jelena},
  booktitle = {Conference on Learning Theory},
  pages = {1428--1451},
  year = 2020,
  publisher = {PMLR}
}

@inproceedings{diakonikolas2021efficient,
  title = {Efficient {{Methods}} for {{Structured Nonconvex-Nonconcave Min-Max Optimization}}},
  author = {Diakonikolas, Jelena and Daskalakis, Constantinos and Jordan, Michael I.},
  booktitle = {Proceedings of {{The}} 24th {{International Conference}} on {{Artificial Intelligence}} and {{Statistics}}},
  pages = {2746--2754},
  year = 2021,
  publisher = {PMLR}
}

@inproceedings{foster2019complexity,
  title = {The Complexity of Making the Gradient Small in Stochastic Convex Optimization},
  author = {Foster, Dylan J. and Sekhari, Ayush and Shamir, Ohad and Srebro, Nathan and Sridharan, Karthik and Woodworth, Blake},
  booktitle = {Conference on {{Learning Theory}}},
  pages = {1319--1345},
  year = 2019,
  publisher = {PMLR}
}

@inproceedings{golowich2020last,
  title = {Last Iterate Is Slower than Averaged Iterate in Smooth Convex-Concave Saddle Point Problems},
  author = {Golowich, Noah and Pattathil, Sarath and Daskalakis, Constantinos and Ozdaglar, Asuman},
  booktitle = {Conference on {{Learning Theory}}},
  pages = {1758--1784},
  year = 2020,
  publisher = {PMLR}
}

@article{goodfellow2014generative,
  title = {Generative Adversarial Nets},
  author = {Goodfellow, Ian J. and {Pouget-Abadie}, Jean and Mirza, Mehdi and Xu, Bing and {Warde-Farley}, David and Ozair, Sherjil and Courville, Aaron and Bengio, Yoshua},
  journal = {Advances in neural information processing systems},
  volume = {27},
  year = 2014
}

@inproceedings{gorbunov2022extragradient,
  title = {Extragradient Method: {{O}} (1/k) Last-Iterate Convergence for Monotone Variational Inequalities and Connections with Cocoercivity},
  author = {Gorbunov, Eduard and Loizou, Nicolas and Gidel, Gauthier},
  booktitle = {International {{Conference}} on {{Artificial Intelligence}} and {{Statistics}}},
  pages = {366--402},
  year = 2022,
  publisher = {PMLR}
}

@article{gorbunov2022lastiterate,
  title = {Last-Iterate Convergence of Optimistic Gradient Method for Monotone Variational Inequalities},
  author = {Gorbunov, Eduard and Taylor, Adrien and Gidel, Gauthier},
  journal = {Advances in neural information processing systems},
  volume = {35},
  pages = {21858--21870},
  year = 2022
}

@inproceedings{gorbunov2023convergence,
  title = {Convergence of Proximal Point and Extragradient-Based Methods beyond Monotonicity: The Case of Negative Comonotonicity},
  author = {Gorbunov, Eduard and Taylor, Adrien and Horv{\'a}th, Samuel and Gidel, Gauthier},
  booktitle = {International {{Conference}} on {{Machine Learning}}},
  pages = {11614--11641},
  year = 2023,
  publisher = {PMLR}
}

@article{juditsky2011solving,
  title = {Solving {{Variational Inequalities}} with {{Stochastic Mirror-Prox Algorithm}}},
  author = {Juditsky, Anatoli and Nemirovski, Arkadi and Tauvel, Claire},
  journal = {Stochastic Systems},
  volume = {1},
  number = {1},
  pages = {17--58},
  year = 2011
}

@article{kim2021accelerated,
  title = {Accelerated Proximal Point Method for Maximally Monotone Operators},
  author = {Kim, Donghwan},
  journal = {Mathematical Programming},
  volume = {190},
  number = {1-2},
  pages = {57--87},
  year = 2021
}

@article{kim2026improving,
  title = {Improving the {{Last-Iterate Guarantees}} of {{Anytime Algorithms}} for {{Stochastic Monotone Variational Inequalities}}},
  author = {Kim, Jun-Hyun and Alacaoglu, Ahmet},
  journal = {arXiv preprint arXiv:2609.15257},
  year = 2026
}

@article{korpelevich1976extragradient,
  title = {The Extragradient Method for Finding Saddle Points and Other Problems},
  author = {Korpelevich, Galina M.},
  journal = {Matecon},
  volume = {12},
  pages = {747--756},
  year = 1976
}

@article{lan2012optimal,
  title = {An Optimal Method for Stochastic Composite Optimization},
  author = {Lan, Guanghui},
  journal = {Mathematical Programming},
  volume = {133},
  number = {1},
  pages = {365--397},
  year = 2012
}

@article{lee2021fast,
  title = {Fast Extra Gradient Methods for Smooth Structured Nonconvex-Nonconcave Minimax Problems},
  author = {Lee, Sucheol and Kim, Donghwan},
  journal = {Advances in Neural Information Processing Systems},
  volume = {34},
  pages = {22588--22600},
  year = 2021
}

@article{lee2022semianchored,
  title = {Semi-{{Anchored Multi-Step Gradient Descent Ascent Method}} for {{Structured Nonconvex-Nonconcave Composite Minimax Problems}}},
  author = {Lee, Sucheol and Kim, Donghwan},
  journal = {arXiv preprint arXiv:2105.15042},
  year = 2022
}

@inproceedings{li2021page,
  title = {{{PAGE}}: {{A Simple}} and {{Optimal Probabilistic Gradient Estimator}} for {{Nonconvex Optimization}}},
  author = {Li, Zhize and Bao, Hongyan and Zhang, Xiangliang and Richtarik, Peter},
  booktitle = {Proceedings of the 38th {{International Conference}} on {{Machine Learning}}},
  pages = {6286--6295},
  year = 2021,
  publisher = {PMLR}
}

@inproceedings{li2022rootsgd,
  title = {Root-Sgd: {{Sharp}} Nonasymptotics and Asymptotic Efficiency in a Single Algorithm},
  author = {Li, Chris Junchi and Mou, Wenlong and Wainwright, Martin and Jordan, Michael},
  booktitle = {Conference on {{Learning Theory}}},
  pages = {909--981},
  year = 2022,
  publisher = {PMLR}
}

@incollection{littman1994markov,
  title = {Markov Games as a Framework for Multi-Agent Reinforcement Learning},
  author = {Littman, Michael L.},
  booktitle = {Machine Learning Proceedings 1994},
  pages = {157--163},
  year = 1994,
  publisher = {Elsevier}
}

@inproceedings{mazumdar2025tractable,
  title = {Tractable Multi-Agent Reinforcement Learning through Behavioral Economics},
  author = {Mazumdar, Eric and Panaganti, Kishan and Shi, Laixi},
  booktitle = {International {{Conference}} on {{Learning Representations}}},
  volume = {2025},
  pages = {18058--18072},
  year = 2025
}

@inproceedings{munos2024nash,
  title = {Nash {{Learning}} from {{Human Feedback}}},
  author = {Munos, Remi and Valko, Michal and Calandriello, Daniele and Azar, Mohammad Gheshlaghi and Rowland, Mark and Guo, Zhaohan Daniel and Tang, Yunhao and Geist, Matthieu and Mesnard, Thomas and Fiegel, C{\^o}me and Michi, Andrea and Selvi, Marco and Girgin, Sertan and Momchev, Nikola and Bachem, Olivier and Mankowitz, Daniel J. and Precup, Doina and Piot, Bilal},
  booktitle = {Proceedings of the 41st {{International Conference}} on {{Machine Learning}}},
  pages = {36743--36768},
  year = 2024,
  publisher = {PMLR}
}

@article{ouyang2021lower,
  title = {Lower Complexity Bounds of First-Order Methods for Convex-Concave Bilinear Saddle-Point Problems},
  author = {Ouyang, Yuyuan and Xu, Yangyang},
  journal = {Mathematical Programming},
  volume = {185},
  number = {1-2},
  pages = {1--35},
  year = 2021
}

@inproceedings{park2022exact,
  title = {Exact Optimal Accelerated Complexity for Fixed-Point Iterations},
  author = {Park, Jisun and Ryu, Ernest K.},
  booktitle = {International {{Conference}} on {{Machine Learning}}},
  pages = {17420--17457},
  year = 2022,
  publisher = {PMLR}
}

@inproceedings{pethick2022escaping,
  title = {Escaping Limit Cycles: {{Global}} Convergence for Constrained Nonconvex-Nonconcave Minimax Problems},
  author = {Pethick, Thomas and Latafat, Puya and Patrinos, Panos and Fercoq, Olivier and Cevher, Volkan},
  booktitle = {International Conference on Learning Representations},
  year = 2022
}

@inproceedings{pethick2023solving,
  title = {Solving Stochastic Weak {{Minty}} Variational Inequalities without Increasing Batch Size},
  author = {Pethick, Thomas and Fercoq, Olivier and Latafat, Puya and Patrinos, Panagiotis and Cevher, Volkan},
  booktitle = {The Eleventh International Conference on Learning Representations},
  year = 2023
}

@article{popov1980modification,
  title = {A Modification of the {{Arrow-Hurwicz}} Method for Search of Saddle Points},
  author = {Popov, L. D.},
  journal = {Mathematical Notes of the Academy of Sciences of the USSR},
  volume = {28},
  number = {5},
  pages = {845--848},
  year = 1980
}

@article{rosen1965existence,
  title = {Existence and Uniqueness of Equilibrium Points for Concave N-Person Games},
  author = {Rosen, J. Ben},
  journal = {Econometrica: Journal of the Econometric Society},
  pages = {520--534},
  year = 1965,
  publisher = {JSTOR}
}

@article{ryu2020ode,
  title = {{{ODE Analysis}} of {{Stochastic Gradient Methods}} with {{Optimism}} and {{Anchoring}} for {{Minimax Problems}}},
  author = {Ryu, Ernest K. and Yuan, Kun and Yin, Wotao},
  journal = {arXiv preprint arXiv:1905.10899},
  year = 2020
}

@inproceedings{sohrabi2026accelerated,
  title = {Accelerated and {{Stable Convergence}} with {{Anchored Generalized Optimistic Method}}},
  author = {Sohrabi, Motahareh and You, Jianxin and {Lacoste-Julien}, Simon and Gorbunov, Eduard and Gidel, Gauthier},
  booktitle = {Forty-Third {{International Conference}} on {{Machine Learning}}},
  year = 2026
}

@article{song2020optimistic,
  title = {Optimistic {{Dual Extrapolation}} for {{Coherent Non-monotone Variational Inequalities}}},
  author = {Song, Chaobing and Zhou, Zhengyuan and Zhou, Yichao and Jiang, Yong and Ma, Yi},
  journal = {Advances in Neural Information Processing Systems},
  volume = {33},
  pages = {14303--14314},
  year = 2020
}

@article{suh2023continuoustime,
  title = {Continuous-Time Analysis of Anchor Acceleration},
  author = {Suh, Jaewook and Park, Jisun and Ryu, Ernest},
  journal = {Advances in Neural Information Processing Systems},
  volume = {36},
  pages = {32782--32866},
  year = 2023
}

@article{surina2026improved,
  title = {An {{Improved Last-Iterate Convergence Rate}} for {{Anchored Gradient Descent Ascent}}},
  author = {Surina, Anja and Suggala, Arun and Tsoukalas, George and Kovsharov, Anton and Shirobokov, Sergey and Ruiz, Francisco J. R. and Kohli, Pushmeet and Chaudhuri, Swarat},
  journal = {arXiv preprint arXiv:2604.03782},
  year = 2026
}

@article{tran-dinh2021halperntype,
  title = {Halpern-{{Type Accelerated}} and {{Splitting Algorithms For Monotone Inclusions}}},
  author = {{Tran-Dinh}, Quoc and Luo, Yang},
  journal = {arXiv preprint arXiv:2110.08150},
  year = 2021
}

@article{tseng2000modified,
  title = {A {{Modified Forward-Backward Splitting Method}} for {{Maximal Monotone Mappings}}},
  author = {Tseng, Paul},
  journal = {SIAM Journal on Control and Optimization},
  volume = {38},
  number = {2},
  pages = {431--446},
  year = 2000
}

@inproceedings{xu2025robust,
  title = {Robust {{LLM}} Alignment via Distributionally Robust Direct Preference Optimization},
  author = {Xu, Zaiyan and Vemuri, Sushil and Panaganti, Kishan and Kalathil, Dileep and Jain, Rahul and Ramachandran, Deepak},
  booktitle = {The Thirty-Ninth Annual Conference on Neural Information Processing Systems},
  year = 2025
}

@inproceedings{yoon2021accelerated,
  title = {Accelerated Algorithms for Smooth Convex-Concave Minimax Problems with {{O}} (1/K\textasciicircum{} 2) Rate on Squared Gradient Norm},
  author = {Yoon, TaeHo and Ryu, Ernest K.},
  booktitle = {International Conference on Machine Learning},
  pages = {12098--12109},
  year = 2021,
  publisher = {PMLR}
}

@article{yoon2025accelerated,
  title = {Accelerated {{Minimax Algorithms Flock Together}}},
  author = {Yoon, TaeHo and Ryu, Ernest K.},
  journal = {SIAM Journal on Optimization},
  volume = {35},
  number = {1},
  pages = {180--209},
  year = 2025
}

@article{yoon2026direct,
  title = {Direct {{Acceleration}} of {{Stochastic Root-Finding Without Variance Reduction}} and {{Regularization}}},
  author = {Yoon, TaeHo and Loizou, Nicolas},
  journal = {arXiv preprint arXiv:2608.12043},
  year = 2026
}
